\documentclass[12pt]{amsart} 
\usepackage{mathrsfs}
\usepackage{amssymb,amsmath,amsthm,color}
\usepackage{graphicx}
\usepackage{hyperref}
\usepackage{url} 
\usepackage{setspace}
\usepackage{mathtools}
\usepackage[inline]{enumitem}  
\usepackage[margin=1in]{geometry}
\usepackage{comment}

\newtheorem{thm}{Theorem}[section]

\newtheorem{rem}[thm]{Remark}
\numberwithin{equation}{section}

\newcommand{\R}{\mathbb R}

\def\TagOnRight

\def\R {\mathbb{R}}

\newcommand{\be}{\begin{equation}}
\newcommand{\ee}{\end{equation}}
\newcommand{\bea}{\begin{eqnarray}}
\newcommand{\eea}{\end{eqnarray}}
\newcommand{\Bea}{\begin{eqnarray*}}
\newcommand{\Eea}{\end{eqnarray*}}
\newcommand{\bt}{\begin{Theorem}}
\newcommand{\et}{\end{Theorem}}

\newcommand{\bpr}{\begin{Proposition}}
\newcommand{\epr}{\end{Proposition}}

\newcommand{\bl}{\begin{Lemma}}
\newcommand{\el}{\end{Lemma}}
\newcommand{\bi}{\begin{itemize}}
\newcommand{\ei}{\end{itemize}}

\newtheorem{Definition}{Definition}[section]
\newtheorem{Theorem}[Definition]{Theorem}
\newtheorem{Lemma}[Definition]{Lemma}
\newtheorem{Proposition}[Definition]{Proposition}
\newtheorem{Corollary}[Definition]{Corollary}
\newtheorem{Remark}[Definition]{Remark}

\usepackage{mathrsfs}
\usepackage{amsthm}
\usepackage{hyperref}
\usepackage{comment}

\begin{document}
\baselineskip16pt
\title[Global Cauchy problems for the HNLKG, HNLW and HNLS]{Global Cauchy problems for the Klein-Gordon, wave and fractional Schr\"odinger equations with Hartree   nonlinearity   on modulation spaces}
\author{Divyang G. Bhimani}
\address{Department of Mathematics\\
University of Maryland\\
College Park\\
MD 20742}
\email{dbhimani@math.umd.edu}


\maketitle
\begin{abstract}  We study  Cauchy problem for the  Klein-Gordon (HNLKG), wave (HNLW) and  Schr\"odinger (HNLS)  equations  with cubic convolution (Hartree type) nonlinearity.   
Some global well-posedness and  scattering  are obtained for the (HNLKG) and (HNLS)   with small Cauchy data in some modulation spaces. Global well-posedness for fractional  Schr\"odinger (fNLSH) equation with Hartree type nonlinearity is obtained with  Cauchy data in some modulation spaces. Local well-posedness for (HNLW), (fHNLS)  and (HNLKG) with rough data in modulation spaces  is shown.  This improves  known results in Sobolev spaces in some sense.   As a consequence, we get local and global well-posedness and scattering in  larger than usual $L^p-$Sobolev spaces and we could include wider class of Hartree type nonlinarity. 
 
\end{abstract}
\section{Introduction and statement of results}
The Cauchy problem (local and global existence, regularity and scattering theory) for (HNLKG), (HNLW) and (fHNLS) has been extensively studied with Cauchy data in  $L^2-$based Sobolev spaces (see e.g., \cite{cazenave2003semilinear, hidano2000small, menzala1982wave, miao2011global, miao2014energy, miao2015scattering, cheng2012small, cho2013cauchy, mochizuki1989small, tsutaya2014scattering}).
There has been a  considerable mathematical interest  concerning the low regularity well-posedness  and scattering theory for the nonlinear  dispersive equations.  Generally the Cauchy data in a modulation spaces are rougher than any given one in a fractional Bessel potential space  and this low regularity is desirable in many situations.

This section is divided into two subsections. Sections \ref{ikw} and \ref{shte} will present well-posedness and scattering  results for  (HNLKG)/(HNLW) and (fHNLS)  with Cauchy data in modulation spaces respectively.

\subsection{Klein-Gordon-Hartree and wave-Hartree equations }\label{ikw}
We study the Cauchy problem for the   Klein-Gordon  and wave equations  with Hartree type  nonliearity 
\begin{equation}\label{nlkgh}
u_{tt}+ (I-\Delta)u = (V\ast |u|^{2})u, 
u(0)=u_0, u_t(0)=u_1
\end{equation}
and
\begin{equation}
\label{nlw}
 u_{tt}- \Delta u = (V\ast |u|^{2}), ~u(0)=u_{0}, u_{t}(0)=u_{1}, 
\end{equation}
where $u(t,x)$ is a complex valued function of $(t,x)\in \mathbb R \times \mathbb R^d$,  $i=\sqrt{-1},$  $u_t=\frac{\partial}{\partial t}, u_{tt}=\frac{\partial^2}{\partial^2 t},$ $I$ is the identity operator, $\Delta$ is the Laplace operator,   $u_0$ and $u_1$ are complex valued functions of $x\in \mathbb R^d,$ $\ast$ denotes the convolution in $\mathbb R^d,$ and  
$V$ is of the following type:
\begin{eqnarray}\label{hk}
\   \ \ \ \ \ \ \ \ \ \ \ \  \ \ \ \   \   \  \   \ V(x)= \frac{\lambda}{|x|^{\gamma}}, (\lambda \in \mathbb R, x\in \mathbb R^{d}, 0<\gamma<d).
\end{eqnarray} 
The stationary equations
\begin{eqnarray*}\label{sq}
-\Delta u + (V\ast |u|^2)u= \sigma u
\end{eqnarray*}
is obtained by looking for separated solutions of \eqref{nlkgh} and \eqref{nlw}, where $u=e^{i\lambda t}u(x)$($\sigma = \lambda^2 -1$ and $\sigma = \lambda^2$). In case  $V(x)=|x|^{-1}$,   the stationary equations was proposed by Hartree as a model for the helium atom. Thus the homogeneous kernel of the form \eqref{hk} is known as Hartree potential. 

  Menzala-Strauss \cite{menzala1982wave} have studied the well-posdedness and asymptotic behavior of  equations \eqref{nlkgh} and \eqref{nlw}. Mochizuki  \cite{mochizuki1989small}  and Hidano \cite{hidano2000small} have studied scattering  theory  in the  energy space (see also \cite{cheng2012small, tsutaya2014scattering}). Recently Miao-Zhang \cite{miao2011global, miao2014energy} and Miao-Zhang-Zheng \cite{miao2015scattering} 
 have  studied  global well-posedness and scattering theory for equations \eqref{nlkgh} and \eqref{nlw} below energy space.  We remark that  all previous authors have studied equations \eqref{nlkgh} and \eqref{nlw} on $L^2-$based Sobolev spaces.  Mainly because  generally Klein-Gordon $ G(t)=e^{it(I-\Delta)^{1/2}}$ and wave $W(t)=e^{it(-\Delta)^{1/2}}$ semigroups   fails to be bounded  on $L^p(\R^d)$ if $p\neq2.$ Hence we cannot expect to solve  equations \eqref{nlkgh} and \eqref{nlw} in  $L^p(\R^d) (p\neq 2)-$spaces .   The question arises if it is possible to remove $L^2$ constraint and consider equations \eqref{nlkgh}
and  \eqref{nlw}  in function spaces which are not $L^2$ based.

 This question has inspired to study  equations \eqref{nlkgh}  and \eqref{nlw} in other function spaces (e.g., modulation spaces  $M^{p,q}(\mathbb R^d)$, see Defintion \ref{ms} below) arising in harmonic analysis. Pioneering  steps in this direction was taken by Baoxiang-Lifeng-Boling \cite{baoxiang2006isometric},  Baoxiang-Hudzik \cite{wang2007global} and   B{\'e}nyi-Gr{\"o}chenig-Okoudjou-Rogers  \cite{benyi2007unimodular}. In fact,  in \cite{wang2007global}  it is  proved that  Klein-Gordon  equation with power type nonlinarity is globally well-posed  with  small Cauchy data in $M^{2,1}(\mathbb R^d).$   In  \cite{benyi2007unimodular, baoxiang2006isometric}  it is proved that the Fourier multiplier operator with multiplier  $e^{it|\xi|^{\alpha}}  \ (\alpha \in [0,2])$ is bounded on  $M^{p,q}(\R^d) \ (1\leq p,q \leq \infty).$  (The cases  $\alpha =1$ and  $\alpha=2$  occurs in the time evolution of the free  wave and Schr\"odinger  equations respectively.)
  Many authors \cite{ruzhansky2012modulation, chen2017non,  benyi2009local, cordero2009remarks,  zhao2014klein,  kato2016solutions, bhimani2016functions}  have studied Klein-Gordon  and wave equations with power type  nonliterary in  modulation spaces.  
However,   there is not much  progress concerning  well-posedness and scattering theory  for the equations \eqref{nlkgh} and \eqref{nlw} in modulation spaces.

Taking these considerations into account, we are inspired  to study  equations \eqref{nlkgh} and \eqref{nlw}   with Cauchy data in modulation spaces.  To sate results,  we set up notations.  Set  $2 \sigma (p)= (d+2) \left( \frac{1}{2} - \frac{1}{p}\right) \ (2<p< \infty, d\in \mathbb N), 1/p+1/p'=1.$ 
We call pair $(p,r)$ is   \textbf{ Klein-Gordon admissible}  if there exists another exponent $\beta$ such that 
\begin{eqnarray}\label{kga}
\begin{cases}  \frac{1}{\beta} + \frac{2}{r}=1,\\
\frac{1}{3} \leq \frac{1}{\beta} \leq \frac{d}{d+2}\wedge d \left( \frac{1}{2} - \frac{1}{p} \right),\\
 \frac{1}{4} \leq p < \frac{1}{2}- \frac{1}{3d}.
\end{cases}
\end{eqnarray}
 We remark that if pair $(p,r)$ is  Klein-Gordon admissible, then  $3\leq r < \infty$ and $rd \left( \frac{1}{2} - \frac{1}{p} \right)>1$.
We are now ready to  state  following  theorem.
\begin{Theorem}[Global well-posedness]\label{mt} Let $2<p<3,  \frac{1}{p}+\frac{\gamma}{d}-1= \frac{1}{2p'},  \ s\in \R,$ and  pair $(p,r)$ is Klein-Gordon admissible.  
Assume that $(u_0, u_1)\in  M^{p',1}_{s+ 2\sigma(p)}(\R^d) \times M^{p',1}_{s+2 \sigma(p)-1}(\R^d)$ and there exists a small $\delta>0$ such that $\|u_0\|_{M^{p',1}_{s+ 2\sigma(p)}} + \|u_1\|_{ M^{p',1}_{s+2 \sigma(p)-1}} \leq \delta.$  Then  \eqref{nlkgh} has a unique  global solution 
\[ u \in  C(\R, M^{p,1}_s(\R^d))\cap C^{1}(\R, M_{s-1}^{p,1}(\R^d)) \cap L^{r}(\R, M_s^{p,1}(\R^d)).\]
One also has the bound
 $\|u\|_{ L^{r}(\R, M_s^{p,1}(\R^d))} \lesssim \|u_0\|_{M^{p',1}_{s+ 2\sigma(p)}} + \|u_1\|_{ M^{p',1}_{s+2 \sigma(p)-1}}.$
\end{Theorem}
Noticing  $L^{p}_s(\mathbb R^d) \subset M^{p,1}(\mathbb R^d)$ for $s>d$ and  taking  $s= -2 \sigma(p)$ (see Theorem \ref{msk} below),  Theorem \ref{mt}
 reveals that we can  control initial   Cauchy data beyond $L^p_s-$Sobolev spaces.   To prove  Theorem \ref{mt} we use some algebraic properties (see Proposition \ref{gap} below)   and the integrability of time decay   terms  for Klein-Gordon semigroup: 
 \[  \|G(t) f\|_{M^{p,q}_s} \lesssim (1+ |t|)^{-d\theta  (1/2-1/p)} \|f\|_{M^{p',q}_{s+ \theta 2 \sigma(p)}},  \]
where $s\in \mathbb R, 2\leq p \leq \infty, 1 \leq q < \infty, \theta \in [0,1]$ (see Proposition \ref{wp} below).  We remark that there is no singularity at $t=0$ and but preserve the same decay as in  the  below  $L^{p}-L^{p'}$ estimate of $G(t).$  This is  a  special  characteristic of modulation spaces. Recall standard $L^{p}-L^{p'}$ estimate of $G(t):$ \[ \|G(t)\|_{L^{p}_{2\sigma(p)}} \leq  C |t|^{-d \left(1/2-1/p\right)} \|f\|_{L^{p'}}   \  \   \ (2\leq p < \infty) \]
and since $|t|^{-d  \left(1/2-1/p\right)}$ is not integrable,  we do  not know whether we can use the similar argument under $L^p,$ Besov, or  Sobolev spaces.

Theorem \ref{mt} reveals that we have   $L^{r}_t(\mathbb R, M^{p,1}_s)$ bound for the solution of \eqref{nlkgh} if the initial data is small enough. This implies we obtain  scattering. Specifically, we have 
\begin{Corollary}[Scattering] \label{mts} Let $u_0 \in M^{p,1}_s(\R^d), u_1 \in M^{p,1}_{s-1}(\R^d),$ and let $u$ is the global solution to \eqref{nlkgh} such that
$ \|u\|_{L^{r}_t(\mathbb R, M^{p,1}_s)} \leq M $
for some  constant $M>0.$  Then there exist $v_{1}^{\pm},v_{2}^{\pm} \in M^{p,1}_s(\R^d)$ such that $v^{\pm}= G(t)v_{1}^{\pm} + G(t)v_2^{\pm}$ are solutions  to the free Klein-Gordon equation $u_{tt} + (I-\Delta)u=0$ and 
\[\|u(t)-v^{\pm}\|_{M^{p,1}_s} \to 0 \  \text{as} \  t \to  \pm \infty. \]
\end{Corollary}
 It remains  open   to get the global well-posedness for equations  \eqref{nlkgh}  and \eqref{nlw} and   for the large data in modulation spaces.  However, we  can get local existence with persistency   of solutions.  Specifically, we have 
\begin{Theorem}[Local wellposedness]\label{LW}
 Let  $V$ is given by \eqref{hk} and   $ X=M^{p,q}(\mathbb R^d)  ( 1\leq p \leq 2, 1\leq q < \frac{2d}{d+\gamma})$ or  $M_s^{p,1}(\mathbb R^d) \ (1<p<\infty, s \in \mathbb R,  \frac{1}{p}+\frac{\gamma}{d}-1= \frac{1}{p+\epsilon}, \epsilon>0).$  Assume that $u_{0}, u_{1}\in X.$ Then 
\begin{enumerate}
\item  \label{lw1}   there exists $T^{*}=T^{\ast}(\|u_{0}\|_{X},\|u_{1}\|_{X})$ such that \eqref{nlkgh} has a unique solution $u\in C([0, T^{*}), X).$  Moreover, if $T^{\ast} < \infty$, then 
$\limsup_{t\to T^{\ast}} \|u(\cdot, t)\|_{X} = \infty.$
\item \label{lw2}  there exists $T^{*}=T^{\ast}(\|u_{0}\|_{X},\|u_{1}\|_{X})$ such that \eqref{nlw} has a unique solution $u\in C([0, T^{*}), X).$  Moreover, if $T^{\ast} < \infty$, then 
$\limsup_{t\to T^{\ast}} \|u(\cdot, t)\|_{X} = \infty.$
\end{enumerate}
\end{Theorem}
Up to now we cannot know equations \eqref{nlkgh}
and \eqref{nlw} are locally well posed in $L^{p}(\mathbb R^d)$ but, by Theorem \ref{LW},  in $M^{p,1}(\mathbb R^d)\subset L^p(\R^d)$ (see Lemma \ref{rl} \eqref{el} below). $M^{p,1}_{s_1}(\R^d)$ ($ p\geq 2,$ some  $s_1\in \mathbb R$)  contains a class of data  which  are out of control of $H^{s}(\R^d).$  Notice that taking $s_{1}=-d/2,$ it follows that  $H^{s}(\R^d)=L^2_s(\R^d) \subsetneq M^{2,1}_{s_1}(\R^d) \subsetneq M^{p,1}_{s_1}(\R^d)$ for any $s>0$ (see Theorem \ref{msk}), Theorem \ref{LW} reveals that  we can get local well-posedness for \eqref{nlkgh} and \eqref{nlw} below energy spaces and in any dimension.   
\begin{Remark} The analogue of Theorem \ref{LW}  holds  for the generalized  equations \eqref{nlkgh} and \eqref{nlw},  that is, Klein-Gordon and wave equations with nonlinearity   $(V\ast |u|^{2k})u \ (k\in \mathbb N)$ when $X= M^{p,1}_s(\R^d).$ 
\end{Remark}
\subsection{Fractional Hartree equation}\label{shte}
We study fractional  Schr\"odinger equation with cubic convolution nonlinearity
\begin{eqnarray}\label{fHTE}
i\partial_t u - (-\Delta)^{\frac{\alpha}{2}}u= (V\ast |u|^2) u, u(x,0)= u_0(x)   
\end{eqnarray}
where $u:\mathbb R_t \times \mathbb R_x^d \to \mathbb C, u_0:\mathbb R^d \to \mathbb C,$  $V$ is defined by \eqref{hk},   and  $\alpha>0.$ The fractional Laplacian is defined as 
\begin{eqnarray*}
\mathcal{F}[(-\Delta)^{\alpha/2}u] (\xi) = |\xi|^{\alpha} \mathcal{F}u (\xi)
\end{eqnarray*}
where $\mathcal{F}$ denotes the Fourier transform.  The equation \eqref{fHTE}  is known as the fractional Hartree  equation. Equation \eqref{fHTE} describes the dynamics of  Bose-Einstein condenstate, in which all particles are  in the same state  $u(t,x).$ There is an extensive  study of \eqref{fHTE} with Cauchy data in Sobolev spaces,  see  e.g.,   \cite{miao2007global, cho2013cauchy, cazenave2003semilinear}  and the references therein.
   
    Recently,  for $0<\gamma < \min \{ \alpha, \frac{d}{2}\},$   Bhimani    
\cite{Bhimani2018global}  proved global well-posedness for \eqref{fHTE} in $M^{p,q}(\mathbb R^{d}) (1\leq p \leq 2, 1\leq q < 2d/ (d+\gamma))$   when $\alpha =2, d\geq 1$,  and with radial Cauchy data when  $d\geq 2, \frac{2d}{2d-1}< \alpha < 2$ (cf. \cite{bhimani2016cauchy, manna2017modulation}).  Manna  \cite{manna2018cauchy}  proved small data global well-posedness for \eqref{fHTE} with the potential  $V\in M^{1, \infty}(\R^d).$  On the other hand, 
  many authors   \cite{baoxiang2006isometric,  wang2007global, benyi2009local, guo2014stability, bhimani2016functions} have studied nonlinear Schr\"odinger equation in modulation spaces.
In this paper, using time integrablity of time decay factors of time decay estimate (see Proposition \ref{uf}), we  obtain global well-posedness and scattering   for  small Cauchy data in modulation spaces.   To state result, we set up notations. We call  pair \textbf{ $(p,r)$ Schr\"odinger admissible}  if there exists another exponent $\beta$ such that 
\begin{eqnarray}\label{sra}
\begin{cases}  \frac{1}{\beta} + \frac{2}{r}=1,\\
\frac{1}{3} \leq \frac{1}{\beta} \leq 1\wedge d \left( \frac{1}{2} - \frac{1}{p} \right),\\
 \frac{1}{4} \leq p < \frac{1}{2}- \frac{1}{3d},
\end{cases}
\end{eqnarray}
and 
\[ (p,r)\neq \left( \frac{2d}{d-2}, \infty \right). \]
 Notice that if  pair $(p,r)$ is Schr\"odinger  admissible, then  $3\leq r \leq  \infty$ and $rd \left( \frac{1}{2} - \frac{1}{p} \right)>1$.
We are now ready to  state  following  theorem.
\begin{Theorem}[Global well-posedness]\label{sdhte} Let $2<p<3,  \frac{1}{p}+\frac{\gamma}{d}-1= \frac{1}{2p'},  \ s\in \R, \alpha =2$ and $(p,r)$  is Schr\"odinger admissible pair.  
Assume that $u_0\in  M^{p',1}_{s}(\R^d) $ and there exists a small $\delta>0$ such that $\|u_0\|_{M^{p',1}_{s}}  \leq \delta.$  Then  \eqref{fHTE} has a unique  global solution 
\[ u \in  C(\R, M^{p,1}_s(\R^d)) \cap L^{r}(\R, M_s^{p,1}(\R^d)).\]
One also has the bound
 $\|u\|_{ L^{r}(\R, M_s^{p,1}(\R^d))} \lesssim \|u_0\|_{M^{p',1}_{s}}.$
\end{Theorem}

 In \cite[Theorem 1.1]{Bhimani2018global}  global well-posedness  for \eqref{fHTE}  studied with the range of $\gamma<  \min \{d/2, 2\}.$ Notice that  Theorem  \ref{sdhte} covers range of $\gamma>d/2$ as $\frac{\gamma}{d}= 1+ \frac{p-3}{2p}$ and  $\frac{\gamma}{d}>\frac{1}{2}\Leftrightarrow p>\frac{3}{2}.$

\begin{Corollary}[Scattering]\label{ssdhte} Let $u_0 \in M^{p,1}_s(\R^d)$ and let $u$ is the global solution to \eqref{fHTE}  with initial $u(0)=u_0$ such that
$ \|u\|_{L^{r}_t(\mathbb R, M^{p,1}_s)} \leq M $
for some  constant $M>0$ and $r< \infty.$  Then there exist  solutions $e^{it\Delta}u_{\pm}$ 
to the free Schr\"odinger equation 
$iu_t+\Delta u =0$
such that  
\[\|u(t)- e^{it\Delta}u_{\pm}\|_{M^{p,1}_s} \to 0 \  \text{as} \  t \to \infty. \]
\end{Corollary}
\begin{Remark} Taking Proposition  \ref{uf} into account,  the method of proof of Theorem \ref{sdhte} may further be applied to  equation   \eqref{fHTE} with $\alpha> 2$ to obtain the  global well-posedness for the small data in modulation spaces.
\end{Remark}

\begin{Theorem}[Global well-posedness]\label{scft} Let $V\in M^{\infty,1}(\R^d)$ and $\frac{1}{2}< \alpha \leq 2.$
 Assume that $u_{0}\in M^{p,q}(\mathbb R^{d}) (1\leq p, q\leq 2). $   Then
 there exists a  unique global solution of \eqref{fHTE} such that $u\in C(\mathbb R, M^{p,q}(\mathbb R^{d})).$
\end{Theorem}
In \cite[Theorem 1.2]{manna2017modulation} it is proved that \eqref{fHTE} with potential $V\in M^{\infty,1}(\R^d)$ and $\alpha=2$ is globally well-posed in $M^{p,q}(\R^d) (1\leq q \leq p \leq 2).$ Notice that Theorem \ref{scft}  generalize this result for \eqref{fHTE} with $\frac{1}{2}< \alpha < 2.$ 

Up to now we cannot know \eqref{fHTE} is locally well-posed in $L^{p}(\R^d)$ but, by Theorem \ref{lwhte}, in $M^{p,1}(\R^d).$   Local well-posedness for \eqref{fHTE}  are studied by many authors in Sobolev spaces.
Modulation spaces  enjoy lower derivative regularity (see Proposition \ref{msk} below) and we can solve \eqref{fHTE}   with the lower regularity assumption for the Cauchy data. Specifically, we have 
\begin{Theorem}[Local well-posedness]\label{lwhte}
 Let  $V$ is given by \eqref{hk},  $\frac{1}{2}< \alpha \leq 2 $ and  $u_0\in M_s^{p,1}(\mathbb R^d) \ (1<p<\infty, s \in \mathbb R,  \frac{1}{p}+\frac{\gamma}{d}-1= \frac{1}{p+\epsilon}, \epsilon>0).$  Then there exists $T^{*}=T^{\ast}(\|u_{0}\|_{M^{p,1}_s})$ such that \eqref{fHTE} has a unique solution $u\in C([0, T^{*}), M^{p,1}_s(\R^d)).$  Moreover, if $T^{\ast} < \infty$, then 
$\limsup_{t\to T^{\ast}} \|u(\cdot, t)\|_{M^{p,1}_s} = \infty.$
\end{Theorem}

\begin{Remark}\label{frd} \
\begin{enumerate}
\item The analogue of Theorem \ref{lwhte}  holds  for the generalized  equation \eqref{fHTE} and \eqref{nlw},  that is, fractional Schr\"odinger equation  with nonlinearity   $(V\ast |u|^{2k})u \ (k\in \mathbb N)$ when $X= M^{p,1}_s(\R^d).$ 

\item \label{frd2} We  have obtain  local well-posedness   for  generalized  equations \eqref{nlkgh},  \eqref{nlw}  and \eqref{fHTE} with potential  $V \in \mathcal{F}L^q(\mathbb R^d) \ (1<q< \infty) \ \text{or}  \   M^{\infty,1}(\mathbb R^d) \ \text{or} \ V\in M^{1,\infty}(\mathbb R^d).$  See Theorems \ref{wht1}  and Remark \ref{fr} below.
\end{enumerate}
 \end{Remark}

The remainder of this paper is organized as follows.  In Section \ref{sp}, we introduce notations and preliminaries which will be used in the sequel.  In  Section \ref{nesthn}, we prove  some  Strichartz type estimates  and   boundedness of   Hartree  nonlinearity  in modulation spaces.
In Section \ref{ps1}, we prove Theorems \ref{mt}, \ref{LW}  and Corollary \ref{mts}. In Section \ref{ps2}, we prove Theorems \ref{sdhte}, \ref{scft} and \ref{lwhte}, and Corollary \ref{ssdhte}.  In Section \ref{lwtp}, we  give   sketch proof  of Remark \ref{frd} \eqref{frd2}.
\section{Preliminaries}\label{sp}  
\subsection{Notations} The notation $A \lesssim B $ means $A \leq cB$ for a some constant $c > 0 $, whereas $ A \asymp B $ means $c^{-1}A\leq B\leq cA $ for some $c\geq 1$ and   $a\wedge b = \text{min} \{ a, b\}.$ The symbol $A_{1}\hookrightarrow A_{2}$ denotes the continuous embedding  of the topological linear space $A_{1}$ into $A_{2}.$ The   $L^{p}(\mathbb R^{d})$ norm is denoted by 
 $$\|f\|_{L^{p}}=\left( \int_{\R^d} |f(x)|^p dx \right)^{1/p} \ \  (1\leq p < \infty),$$
the $L^{\infty}(\mathbb R^{d})$ norm  is $\|f\|_{L^{\infty}}= \text{ess.sup}_{ x\in \mathbb R^{d}}|f(x)| $. For $1\leq p\leq \infty,$ $p'$ denotes the H\"older conjugate of $p,$ that is, $1/p+1/p'=1.$  
We use  $L_t^{r} (I, X)$ to denote the space time norm
\[\|u\|_{L^{r}_t(I, X)}=  \left( \int_{I} \|u\|^r_{X}  dt \right)^{1/r},\]
where $I\subset \R$ is an interval and $X$ is a Banach space.
The Schwartz space is denoted by  $\mathcal{S}(\mathbb R^{d})$ (with it's usual topology), and the space of tempered distributions is  denoted by $\mathcal{S'}(\mathbb R^{d}).$ For $x=(x_1,\cdots, x_d), y=(y_1,\cdots, y_d) \in \mathbb R^d, $ we put $x\cdot y = \sum_{i=1}^{d} x_i y_i.$
Let $\mathcal{F}:\mathcal{S}(\mathbb R^{d})\to \mathcal{S}(\mathbb R^{d})$ be the Fourier transform  defined by  
\begin{eqnarray*}
\mathcal{F}f(w)=\widehat{f}(w)=\int_{\mathbb R^{d}} f(t) e^{- 2\pi i t\cdot w}dt, \  w\in \mathbb R^d.
\end{eqnarray*}
Then $\mathcal{F}$ is a bijection  and the inverse Fourier transform  is given by
\begin{eqnarray*}
\mathcal{F}^{-1}f(x)=f^{\vee}(x)=\int_{\mathbb R^{d}} f(w)\, e^{2\pi i x\cdot w} dw,~~x\in \mathbb R^{d},
\end{eqnarray*}
and this Fourier transform can be uniquely extended to $\mathcal{F}:\mathcal{S}'(\mathbb R^d) \to \mathcal{S}'(\mathbb R^d).$
The \textbf{Fourier-Lebesgue spaces} $\mathcal{F}L^p(\mathbb R^d)$ is defined by 
$$\mathcal{F}L^p(\mathbb R^d)= \left\{f\in \mathcal{S}'(\mathbb R^d): \|f\|_{\mathcal{F}L^{p}}:= \|\hat{f}\|_{L^{p}}< \infty \right\}.$$
The standard  Sobolev  spaces $W^{s,p}(\R^d) \ (1<p< \infty, s\geq 0)$ have a different character according to whether $s$ is integer or not. Namely, for $s$ integer, they consist of  $L^p-$functions with derivatives in  $L^p$ up to order $s$, hence coincide with the \textbf{$L^p_s-$Sobolev spaces }(also known as \textbf{Bessel potential
spaces}), defined for  $s\in \R$ by
$$L^p_s(\mathbb R^d)=\left\{f\in \mathcal{S}'(\mathbb R^d): \|f\|_{L^{p}_s}:=\left\| \mathcal{F}^{-1} [\langle \cdot \rangle^s \mathcal{F}(f)]\right\|_{L^p}< \infty \right\},$$
where  $\langle \xi \rangle^{s} = (1+ |\xi|^2)^{s/2} \ (\xi \in \mathbb R^d).$ Note that $L^p_{s_1}(\R^d) \hookrightarrow L^p_{s_2}(\R^d)$ if $s_2\leq s_1.$  

\subsection{Modulation  spaces} Feichtinger  \cite{feichtinger1983modulation} introduced a  class of Banach spaces,  the so called modulation spaces,  which allow a measurement of space variable and Fourier transform variable of a function or distribution on $\mathbb R^d$ simultaneously using the short-time Fourier transform(STFT). The  STFT  of a function $f$ with respect to a window function $g \in {\mathcal S}(\R^d)$ is defined by
\begin{eqnarray*}\label{stft}
V_{g}f(x,w)= \int_{\mathbb R^{d}} f(t) \overline{g(t-x)} e^{-2\pi i w\cdot t}dt,  \  (x, w) \in \mathbb R^{2d}
\end{eqnarray*}
 whenever the integral exists.
For $x, y \in \R^d$ the translation operator $T_x$ and the modulation operator $M_y$ are
defined by $T_{x}f(t)= f(t-x)$ and $M_{y}f(t)= e^{2\pi i y\cdot t} f(t).$ In terms of these
operators the STFT may be expressed as
\begin{eqnarray}
\label{ipform} V_{g}f(x,y)=\langle f, M_{y}T_{x}g\rangle\nonumber
\end{eqnarray}
 where $\langle f, g\rangle$ denotes the inner product for $L^2$ functions,
or the action of the tempered distribution $f$ on the Schwartz class function $g$.  Thus $V: (f,g) \to V_g(f)$ extends to a bilinear form on $\mathcal{S}'(\R^d) \times \mathcal{S}(\R^d)$ and $V_g(f)$ defines a uniformly continuous function on $\R^{d} \times \R^d$ whenever $f \in \mathcal{S}'(\R^d) $ and $g \in  \mathcal{S}(\R^d)$.
\begin{Definition}[modulation spaces]\label{ms} Let $1 \leq p,q \leq \infty, s \in \R$ and $0\neq g \in{\mathcal S}(\R^d)$. The  weighted  modulation space   $M_s^{p,q}(\R^d)$
is defined to be the space of all tempered distributions $f$ for which the following  norm is finite:
$$ \|f\|_{M_s^{p,q}}=  \left(\int_{\R^d}\left(\int_{\R^d} |V_{g}f(x,y)|^{p} dx\right)^{q/p} (1+|y|^2)^{sq/2} \, dy\right)^{1/q},$$ for $ 1 \leq p,q <\infty$. If $p$ or $q$ is infinite, $\|f\|_{M_s^{p,q}}$ is defined by replacing the corresponding integral by the essential supremum. 
\end{Definition}
For $s=0,$ we write $M^{p,q}_0(\R^d)= M^{p,q}(\R^d).$
\begin{rem}
\label{equidm}
The definition of the modulation space given above, is independent of the choice of 
the particular window function.  See  \cite[Proposition 11.3.2(c)]{grochenig2013foundations}.
\end{rem}
Applying the frequency-uniform localization techniques,
one can get an equivalent definition of modulation spaces  \cite{wang2007global}  as follows. Let  $Q_k$ be the unit cube with the center at  $k$, so $\{ Q_{k}\}_{k \in \mathbb Z^d}$ constitutes a decomposition of  $\mathbb R^d,$ that is, $\mathbb R^d = \cup_{k\in \mathbb Z^{d}} Q_{k}.$
Let   $\rho \in \mathcal{S}(\mathbb R^d),$  $\rho: \mathbb R^d \to [0,1]$  be  a smooth function satisfying   $\rho(\xi)= 1 \  \text{if} \ \ |\xi|_{\infty}\leq \frac{1}{2} $ and $\rho(\xi)=
0 \  \text{if} \ \ |\xi|_{\infty}\geq  1.$ Let  $\rho_k$ be a translation of $\rho,$ that is,
\[ \rho_k(\xi)= \rho(\xi -k) \ (k \in \mathbb Z^d).\]
Denote 
$$\sigma_{k}(\xi)= \frac{\rho_{k}(\xi)}{\sum_{l\in\mathbb Z^{d}}\rho_{l}(\xi)}, \ (k \in \mathbb Z^d).$$   Then   $\{ \sigma_k(\xi)\}_{k\in \mathbb Z^d}$ satisfies the following properties
\begin{eqnarray*}
\begin{cases} |\sigma_{k}(\xi)|\geq c, \forall z \in Q_{k},\\
\text{supp} \ \sigma_{k} \subset \{\xi: |\xi-k|_{\infty}\leq 1 \},\\
\sum_{k\in \mathbb Z^{d}} \sigma_{k}(\xi)\equiv 1, \forall \xi \in \mathbb R^d,\\
|D^{\alpha}\sigma_{k}(\xi)|\leq C_{|\alpha|}, \forall \xi \in \mathbb R^d, \alpha \in (\mathbb N \cup \{0\})^{d}.
\end{cases}
\end{eqnarray*}
The frequency-uniform decomposition operators can be exactly defined by 
$$\square_k = \mathcal{F}^{-1} \sigma_k \mathcal{F}. $$
For $1\leq p, q \leq \infty, s\in \mathbb R,$  it is known \cite{feichtinger1983modulation} that 
\begin{eqnarray*}
\|f\|_{M^{p,q}_s}\asymp  \left( \sum_{k\in \mathbb Z^d} \| \square_k(f)\|^q_{L^p} (1+|k|)^{sq} \right)^{1/q},
\end{eqnarray*}
with natural modifications for $p, q= \infty.$
We notice almost orthogonality  relation for the frequency-uniform decomposition operators
 \begin{eqnarray*}
 \square_k= \sum_{\|\ell \|_{\infty}\leq 1} \square_{k+\ell}\square_{k} \ \ (k, \ell \in \mathbb Z^{d})
\end{eqnarray*}
where $\|\ell\|_{\infty}= \max \{|\ell_i|:\ell_i \in \mathbb Z, i=1,..., d\}.$ 
\begin{Lemma} [\cite{wang2011harmonic, grochenig2013foundations,ruzhansky2012modulation}]  \label{rl} Let $p,q, p_{i}, q_{i}\in [1, \infty]$  $(i=1,2), s, s_1, s_2 \in \R.$ Then
\begin{enumerate}
\item \label{ir} $M^{p_{1}, q_{1}}_{s_1}(\mathbb R^{d}) \hookrightarrow M^{p_{2}, q_{2}}_{s_2}(\mathbb R^{d})$ whenever $p_{1}\leq p_{2}$ and $q_{1}\leq q_{2}$ and $s_2\leq s_1.$
\item \label{el} $M^{p,q_{1}}(\mathbb R^{d}) \hookrightarrow L^{p}(\mathbb R^{d}) \hookrightarrow M^{p,q_{2}}(\mathbb R^{d})$ holds for $q_{1}\leq \text{min} \{p, p'\}$ and $q_{2}\geq \text{max} \{p, p'\}$ with $\frac{1}{p}+\frac{1}{p'}=1.$
\item \label{rcs} $M^{\min\{p', 2\}, p}(\mathbb R^d) \hookrightarrow \mathcal{F} L^{p}(\mathbb R^d)\hookrightarrow M^{\max \{p',2\},p}(\mathbb R^d),  \frac{1}{p}+\frac{1}{p'}=1.$
\item $\mathcal{S}(\mathbb R^{d})$ is dense in  $M^{p,q}(\mathbb R^{d})$ if $p$ and $q<\infty.$
\item  $M^{p,p}(\mathbb R^d) \hookrightarrow L^p(\mathbb R^d) \hookrightarrow M^{p,p'}(\mathbb R^d)$ for $1\leq p \leq 2$ and  $M^{p,p'}(\mathbb R^d)  \hookrightarrow L^p(\mathbb R^d) \hookrightarrow M^{p,p}(\mathbb R^d)$ for $2 \leq p \leq \infty.$
\item \label{fi} The Fourier transform $\mathcal{F}:M_s^{p,p}(\mathbb R^{d})\to M_s^{p,p}(\mathbb R^{d})$ is an isomorphism.
\item The space  $M_s^{p,q}(\mathbb R^{d})$ is a  Banach space.
\item \label{ic}The space $M_s^{p,q}(\mathbb R^{d})$ is invariant under complex conjugation.
\end{enumerate}
\end{Lemma}
\begin{Theorem}[\cite{kobayashi2011inclusion, ruzhansky2012modulation}]\label{msk}
Let $1\leq p, q \leq \infty,$ $s_1, s_2 \in \R,$ and 
$$\tau (p,q)= \max \left\{ 0, d\left( \frac{1}{q}- \frac{1}{p}\right), d\left( \frac{1}{q}+ \frac{1}{p}-1\right) \right\}.$$Then 
$L^{p}_{s_1}(\R^d) \subset M^{p,q}_{s_2}(\R^d)$ if and only if one of the following conditions is satisfied:
\begin{enumerate}
\item[(i)] $q\geq p>1, s_1\geq s_2 + \tau(p,q);$
\item[(ii)] $p>q, s_1>s_2+ \tau(p,q);$
\item[(iii)]  $p=1, q=\infty, s_1\geq s_2 + \tau(1, \infty);$
\item[(iv)] $p=1, q\neq \infty, s_1>s_2+\tau (1, q).$
\end{enumerate} 
\end{Theorem}
\begin{Proposition}[Algebra property \cite{benyi2009local}]\label{gap}  Let $m\in \mathbb N, s\geq 0.$  Assume that $\sum_{i=1}^{m} \frac{1}{p_i}= \frac{1}{p_0}, \sum_{i=1}^{m} \frac{1}{q_i}= m-1+ \frac{1}{q_0}$ with $0<p_i\leq \infty, 1\leq q_i \leq \infty$  for $1\leq i \leq m.$ Then  we have 
\[  \left \| \prod_{i=1}^{m} u_i \right\|_{M^{p_0, q_0}_s} \lesssim  \prod_{i=1}^{m} \|u_i\|_{M^{p_i, q_i}_s}.\]
\end{Proposition}
\begin{Proposition}[isomorphism  \cite{feichtinger1983modulation}] \label{iso}Let $0<p,q \leq \infty, s, \sigma \in \mathbb R.$ Then  $J_{\sigma}: (I-\Delta)^{\sigma/2}:M^{p,q}_s(\R^d) \to M^{p,q}_{s-\sigma}(\R^d)$  is an isomorphic mapping. (We denote $J_1=J.$)
\end{Proposition}
\begin{Lemma}\label{dl} Let $s\in \mathbb R, 1\leq p,q < \infty,$ and $\Omega$ be a compact subset of $\mathbb R^d.$  Then $\mathcal{S}^{\Omega}= \{f: f\in \mathcal{S}(\mathbb R^d) \ \text{and supp} \ \hat{f} \subset \Omega \}$ is dense in $M_s^{p,q}(\mathbb R^d).$
\end{Lemma}
For $f\in \mathcal{S}(\mathbb R^{d}),$ we define the fractional Schr\"odinger propagator $e^{it(-\Delta)^{\alpha/2}}$ for $t, \alpha \in \mathbb R$ as follows:
\begin{eqnarray*}
\label{sg}
U(t)f(x)=e^{it (-\Delta)^{\alpha/2}}f(x)= \int_{\mathbb R^d}  e^{i \pi t|\xi|^{\alpha}}\, \hat{f}(\xi) \, e^{2\pi i \xi \cdot x} \, d\xi.
\end{eqnarray*}
When $\alpha=2,$ we write $U(t)=S(t)= e^{-it\Delta}$ (corresponding to usual Schr\"odinger equation).
The next proposition shows that the uniform boundedness and truncated decay estimates  of the Schr\"odinger propagator $e^{it(-\Delta)^{\alpha/2}}$ on modulation spaces.  
\begin{Proposition}[\cite{chen2012asymptotic, wang2007global}]  \
 \label{uf}
 \begin{enumerate}
 \item Let $\frac{1}{2} < \alpha \leq 2, 1 \leq p,q \leq \infty.$ Then $ \|U(t)f \|_{M^{p,q}}\leq  (1+|t|)^{d\left| \frac{1}{p}-\frac{1}{2} \right|} \|f\|_{M^{p,q}}.$
 \item Let $\alpha \geq 2 $ and $ 2 \leq p,q \leq \infty.$ Then $ \|U(t)f \|_{M^{p,q}}\leq  (1+|t|)^{- \frac{2d}{\alpha}\left( \frac{1}{p}-\frac{1}{2} \right)} \|f\|_{M^{p',q}}.$
 \end{enumerate} 
\end{Proposition}
Now we consider the truncated decay estimate and uniform bounded estimates  for the
Klein-Gordon semigroup  $G(t).$
\begin{Proposition}[See Proposition 4.2 in  \cite{wang2007global}]\label{wp} Let $G(t)= e^{it (I-\Delta)^{1/2}} \ (t\in \mathbb R).$
\begin{enumerate}
\item Let $s\in \mathbb R, 2\leq p \leq \infty,$ $1 \leq  q < \infty, \theta \in [0,1],$ and
$2 \sigma (p)= (d+2) \left( \frac{1}{2} - \frac{1}{p}\right).$
Then we have 
\[  \|G(t) f\|_{M^{p,q}_s} \lesssim (1+ |t|)^{-d\theta  (1/2-1/p)} \|f\|_{M^{p',q}_{s+ \theta 2 \sigma(p)}}. \]
\item  \label{kb} Let  $s\in \mathbb R $ and $ 1\leq p, q  \leq \infty.$   Then we have 
\[  \|G(t)f\|_{M^{p,q}_s} \leq C (1+ |t|)^{d \left| 1/2 -1/p \right|}  \|f\|_{M^{p,q}_s}.\]
\end{enumerate}
\end{Proposition}
\begin{Proposition}[Uniform boundedness of wave propagator \cite{benyi2009local}]\label{wpm}  For $\sigma^{1} (\xi) = \sin(2 \pi t|\xi|)/ 2 \pi |\xi|, ~~~$  $\sigma^{2}(\xi)=\cos(2 \pi t|\xi|)$ and $f\in \mathcal{S}(\R^d),$ we define $H_{\sigma^{i}}f(x)= \left( \sigma^{i}\hat{f}\right)^{\vee}(x) \ (x\in \R^d, i=1,2).$ Let $s\in \R$ and $1\leq p, q \leq \infty.$  Then we have 
\[\| H_{\sigma^i } f\|_{M^{p,q}_{s}}  \leq c_d(1+t^{2})^{d/4} \|  f\|_{M^{p,q}_{s}}. \]
\end{Proposition}
\begin{Proposition}[Bernstein multiplier theorem \cite{wang2011harmonic} ] Let $L\in \mathbb Z, L>d/2, \partial_{x_i}^{\alpha} \rho \in L^2, i=1,2, ..., d, 0 \leq \alpha \leq L.$ Then $\rho$ is a multiplier on $L^p$  $(1\leq p \leq \infty).$ Moreover there exists a constant $C$ such that 
\[  \|\rho\|_{M_p} \leq C \|\rho\|_{L^2}^{1-d/2L} \left( \sum_{i=1}^d \|\partial_{x_i}^{L} \rho \|_{L^2} \right) ^{d/2L}.\]
\end{Proposition}
\begin{Proposition}[\cite{wang2011harmonic}]\label{lpm} Let $\Omega \subset \R^d$ be a compact subset  and let $1\leq p \leq \infty,  \ s_{p} = d \left( \frac{1}{p\wedge 1} - \frac{1}{2} \right).$ If $s>s_p,$ then there exists a $C>0$ such that  $\|\mathcal{F}^{-1} \phi \mathcal{F}\phi \|_{L^p} \leq C \|\phi\|_{H^{s}} \|f\|_{L^p}$ holds for all $f\in L^{p}{\Omega}$ and $\phi \in H^{s}(\R^d)= L^2_s(\R^d).$
\end{Proposition}
\section{Nonlinear Estimates in $M^{p,q}_s(\R^d)$}\label{nesthn}
In this section we prove  estimates for Hartree  nonlinearity (Corollary \ref{afi} and Lemmas \ref{cl} and \ref{mcl}) and Strichartz type estimates (Proposition \ref{nel}). We shall apply these    to prove  main theorems in the following sections.

We define fractional integral operator $T_{\gamma} (0< \gamma <d)$ as follows
\[  T_{\gamma} f(x)= V_{\gamma}\ast f (x) =  \pm  \int_{\mathbb R^d} \frac{f(y)}{|x-y|^{\gamma}}  dy, \  \     \  \   \ (f\in \mathcal{S}(\mathbb R^d), V_{\gamma}(x)= \pm |x|^{-\gamma}).  \]
It is known  $T_{\gamma}$ is bounded from $L^{p}(\mathbb R^d)$ to $L^q(\mathbb R^d)$ for some specific  $p,q$ and $\gamma.$  Specifically, we have 
\begin{Proposition}[Hardy-Littlewood-Sobolev inequality] \label{hls}Assume that  $0<\gamma< d$ and $1<p<q< \infty$ with
$\frac{1}{p}+\frac{\gamma}{d}-1= \frac{1}{q}.$
Then we have
$\|T_{\gamma}f\|_{L^q} \leq C_{d,\gamma, p} \|f\|_{L^p}.$
\end{Proposition}
We prove analogue of Hardy-Littlewood-Sobolev inequality in case of modulation spaces.
\begin{Proposition} \label{fip} Assume that $0<\gamma < d, 1< p_1< p_2< \infty $   with
$$\frac{1}{p_1}+\frac{\gamma}{d}-1= \frac{1}{p_2}$$and 
$1\leq  q \leq \infty, s\geq 0.$   Then the map $T_{\gamma}$ is bounded from $M_s^{p_1,q}(\mathbb R^d)$ to $ M_s^{p_2,q}(\mathbb R^d):$
\[\|T_{\gamma}f\|_{M^{p_2,q}_s} \lesssim \|f\|_{M^{p_1,q}_s}.\]
\end{Proposition}
\begin{proof} We may rewrite the STFT as 
$V_{g}(x,w)= e^{-2\pi i  x \cdot w} (f\ast M_wg^*)(x)$  where $g^*(y) = \overline{g(-y)}.$ Using Hardy-Littlewood-Sobolev inequality, we obtain
\begin{eqnarray*}
\|T_{\gamma}f\|_{M^{p_2,q}_s} & =  &  \left \|   \|V_{\gamma} \ast  (f \ast M_{w}g^{*})\|_{L^{p_2}}  \langle w \rangle^s  \right\|_{L^{q}_w}\\
& \lesssim &   \left \|   \|f \ast M_{w}g^{*})\|_{L^{p_1}}  \langle w \rangle^s  \right\|_{L^{q}_w}\\
& \lesssim & \|f\|_{M^{p_1,q}_s}.
\end{eqnarray*}
This completes the proof.
\end{proof}
\begin{Corollary} \label{afi} Let $1< p<\infty $ and $\frac{1}{p}+\frac{\gamma}{d}-1= \frac{1}{p+\epsilon}$
 for some $\epsilon>0.$  Then  $$\|(V_{\gamma}\ast |f|^{2k})f\|_{M_s^{p,1}} \lesssim \|f\|_{M_s^{p,1}}^{2k+1} \  \   \  (k\in \mathbb N).$$ 
\end{Corollary}
\begin{proof}
By Proposition  \ref{gap} and Lemma \ref{rl}\eqref{ir}, we have 
\begin{eqnarray*}
\|(V_{\gamma}\ast |f|^{2k})f\|_{M_s^{p,1}} & \lesssim & \|T_{\gamma} |f|^{2k}\|_{M_s^{\infty,1}} \|f\|_{M_s^{p,1}}\\
& \lesssim & \|T_{\gamma}|f|^{2k}\|_{M_s^{p+\epsilon,1}} \|f\|_{M_s^{p,1}},
\end{eqnarray*}
for some $\epsilon>0.$
By  Propositions \ref{fip} and  \ref{gap}, we have  $ \|T_{\gamma}|f|^{2k}\|_{M_s^{p+\epsilon,1}} \lesssim \||f|^{2k}\|_{M_s^{p,1}} \lesssim \|f\|^{2k}_{M_s^{p,1}}.$  This completes the proof.
\end{proof}
\begin{Lemma}\label{cl}Let $1< p<\infty $ and $\frac{1}{p}+\frac{\gamma}{d}-1= \frac{1}{p+\epsilon}$
for some $\epsilon>0.$ Then we have
$$\| (V_{\gamma}\ast |f|^{2})f - (V_{\gamma}\ast |g|^{2})g\|_{M_s^{p,1}} \lesssim  (\|f\|_{M_s^{p,1}}^{2}+\|f\|_{M_s^{p,1}}\|g\|_{M_s^{p,1}}+ \|g\|_{M_s^{p,1}}^{2}) \|f-g\|_{M_s^{p,1}}.$$
\end{Lemma}
\begin{proof}
Exploiting the ideas of proof as in  Corollary \ref{afi}, we obtain
\begin{eqnarray*}
 \|(V_{\gamma}\ast |f|^{2})(f-g)\|_{M_s^{p,1}}  &  \lesssim & \|f\|_{M_s^{p,1}}^2 \|f-g\|_{M_s^{p,1}},
\end{eqnarray*}
and 
\begin{eqnarray*}
 \|(V_{\gamma} \ast (|f|^{2}- |g|^{2}))g\|_{M_s^{p,1}}  &  \lesssim & \||f|^2-|g|^2\|_{M_s^{p,1}} \|g\|_{M_s^{p,1}}\\
 & \lesssim & \left( \|f\|_{M_s^{p,1}} \|g\|_{M_s^{p,1}} + \|g\|_{M_s^{p,1}}^2\right) \|f-g\|_{M_s^{p,1}}.
 \end{eqnarray*}
 This together with the following identity
$$(V_{\gamma}\ast |f|^{2})f- (V_{\gamma}\ast |g|^{2})g= (V_{\gamma}\ast |f|^{2})(f-g) + (V_{\gamma} \ast (|f|^{2}- |g|^{2}))g, $$
gives  the desired inequality.
\end{proof}
\begin{Lemma}\label{mcl}Let $2< p<2p' $ and $\frac{1}{p}+\frac{\gamma}{d}-1= \frac{1}{2p'}.$
Then we have
$$\| (V_{\gamma}\ast |f|^{2})f - (V_{\gamma}\ast |g|^{2})g\|_{M_s^{p',1}} \lesssim  (\|f\|_{M_s^{p,1}}^{2}+\|f\|_{M_s^{p,1}}\|g\|_{M_s^{p,q}}+ \|g\|_{M_s^{p,1}}^{2}) \|f-g\|_{M_s^{p,1}}.$$
\end{Lemma}
\begin{proof}
By Proposition \ref{gap}, we have 
\begin{eqnarray*}
 \|(V_{\gamma}\ast |f|^{2})(f-g)\|_{M_s^{p',1}}  &  \lesssim & \|V_{\gamma}\ast |f|^2\|_{M_s^{2p',1}}  \|f-g\|_{M_s^{2p',1}}\\
  & \lesssim &  \||f|^2\|_{M^{p,1}_s} \|f-g\|_{M^{p,1}_s}
\end{eqnarray*}
and 
\begin{eqnarray*}
 \|(V_{\gamma} \ast (|f|^{2}- |g|^{2}))g\|_{M_s^{p',1}}  &  \lesssim & \| V_{\gamma} \ast (|f|^2-|g|^2)\|_{M_s^{2p',1}} \|g\|_{M_s^{2p',1}}\\
 &  \lesssim & \||f|^2-|g|^2\|_{M_s^{p,1}} \|g\|_{M_s^{p,1}}\\
 & \lesssim & \left( \|f\|_{M_s^{p,1}} \|g\|_{M_s^{p,1}} + \|g\|_{M_s^{p,1}}^2\right) \|f-g\|_{M_s^{p,1}}.
 \end{eqnarray*}
\end{proof}
Recall that equation \eqref{nlkgh} have the following
equivalent form
\[u(t)= K'(t)u_0 + K(t)u_1 -\mathcal{B}f(u),\]
where we denote 
$\omega= (I-\Delta),$
\[K(t)=\frac{\sin t \omega^{1/2}}{\omega^{1/2}}, \  \  K'(t)= \cos t \omega^{1/2}, \ \  \mathcal{B}= \int_0^t K(t-\tau) \cdot d\tau.\]
We prove following Strichartz type estimates  in modulation spaces.
\begin{Proposition}\label{nel}   Let $ F(u)= (V_{\gamma}\ast |u|^2)u,  p\in (2,3),$ $\frac{1}{p}+\frac{\gamma}{d}-1= \frac{1}{2p'}$  and  pair $(p,r)$ is  Klein-Gordon admissible. Then we have
\begin{eqnarray*}
\left \|  \int_{0}^{t} K(t-\tau) F(u(\tau))  d\tau \right\|_{L^{r}_t (\R, M^{p,1}_s)} &  \lesssim  &\|F(u)\|_{L^{r/3}_t(\R, M^{p',1}_s)}  \lesssim   \|u\|_{L^{r}_t(\R, M^{p,1}_s)}^3.
\end{eqnarray*}
\end{Proposition}
\begin{proof} Since
$G(t)=e^{it\omega^{1/2}}, $  we have  $K(t)\omega^{1/2} = (G(t)-G(-t))/2i.$  By general Minkowski inequality, Propositions \ref{wp} and \ref{iso}, we have 
\begin{eqnarray*}
\left \|  \int_{0}^{t} K(t-\tau) F(u(\tau))  d\tau \right\|_{L^{r}_t (\R, M^{p,1}_s)} & \lesssim & \left \|  \int_{0}^{t} \|K(t-\tau) F(u(\tau)) \|_{M^{p,1}_s} d\tau \right\|_{L^{r}_t(\R)}\\
& \lesssim & \left \|  \int_{0}^{t} (1+ |t-\tau|)^{-d\theta  (1/2-1/p)} \|F(u)\|_{M^{p',1}_{s+ \theta 2 \sigma(p)-1}} d\tau \right\|_{L^{r}_t(\R)}\\
& \lesssim & \left \|  \int_{\mathbb R } (1+ |t-\tau|)^{-d\theta  (1/2-1/p)} h(\tau) d\tau\right\|_{L^{r}_t(\mathbb R)}\\
& \lesssim &  \| g\ast h \|_{L^r_t},
\end{eqnarray*}
where  $ h(\tau)=  \|F(u)\|_{M^{p',1}_{s+ \theta 2 \sigma(p)-1}}, g(t)=(1+ |t|)^{-d\theta  (1/2-1/p)} $ and $\theta \in [0,1].$ We divide Klein-Gordon admissible pairs (see \eqref{kga}) into two cases.
\item \textbf{Case I}: $\frac{1}{\beta}= \frac{d}{d+2}\wedge d \left( \frac{1}{2} - \frac{1}{p} \right).$ In this case  $\frac{1}{\beta} <1$ and there exists $\theta\in (0,1]$ such that 
$$\frac{1}{\beta}= \theta d \left( \frac{1}{2} - \frac{1}{p} \right)=\frac{d}{d+2}\wedge d \left( \frac{1}{2} - \frac{1}{p} \right).$$
With this $\theta,$ we have   $\theta 2 \sigma (p)-1 \leq 0.$ Since pair $(p,r)$ is  Klein-Gordon admissible, we have
\[\frac{1}{r}= \frac{3}{r}- \frac{1- d\theta (1/2-1/p)}{1}\]
and $r/3>1.$ With this $\theta,$ by Hardy-Littlewood-Sobolev inequality in  dimension one, we have 
\begin{eqnarray*}
\left \|  \int_{0}^{t} K(t-\tau) F(u(\tau))  d\tau \right\|_{L^{r}_t (\R, M^{p,1}_s)} 
& \lesssim &  \| g\ast h \|_{L^r_t (\mathbb R)}\\
& \lesssim  &  \left\|\|F(u)\|_{M^{p',1}_{s}}\right\|_{L^{r/3}} = \|F(u)\|_{L^{r/3}(\R, M^{p',1}_s)}.
\end{eqnarray*}
\item \textbf{Case II}: $\frac{1}{\beta}< \frac{d}{d+2}\wedge d \left( \frac{1}{2} - \frac{1}{p} \right).$ In this case there exists $\theta\in [0,1]$ such that 
$$\frac{1}{\beta}< \theta d \left( \frac{1}{2} - \frac{1}{p} \right)\leq \frac{d}{d+2}\wedge d \left( \frac{1}{2} - \frac{1}{p} \right).$$
With this $\theta,$ we have 
$\beta \theta d \left( \frac{1}{2} - \frac{1}{p} \right)>1,$ and $\theta 2 \sigma(p)-1 \leq 0.$  By Young and H\"older inequalities, we have 
\begin{eqnarray*}
\left \|  \int_{0}^{t} K(t-\tau) F(u(\tau)) d\tau  \right\|_{L^{r}_t (\R, M^{p,1}_s)} 
& \lesssim &  \| g\ast h \|_{L^r_t}\\
& \lesssim  & \|g\|_{L^{\beta}} \left\|\|F(u)\|_{M^{p',1}_{s}}\right\|_{L^{r/3}}  \lesssim  \|F(u)\|_{L^{r/3}(\R, M^{p',1}_s)}.
\end{eqnarray*}
By Propositions \ref{gap} and \ref{fip} and Lemma \ref{rl} \eqref{ir},  we have
\begin{eqnarray*}
 \|F(u)\|_{L^{r/3}(\R, M^{p',1}_s)} & \lesssim & \left( \int  (\|T_{\gamma}|u|^2\|_{M^{2p',1}_s} \|u\|_{M_s^{2p',1}} )^{r/3} dt\right)^{3/r} \\
 & \lesssim &  \left( \int  (\||u|^2\|_{M^{p,1}_s} \|u\|_{M_s^{p,1}} )^{r/3} dt\right)^{3/r} \\
  & \lesssim & \left( \int  \|u\|_{M_s^{p,1}} ^{r} dt\right)^{3/r}   \lesssim   \|u\|_{L^{r}_t(\R, M^{p,1}_s)}^3.
\end{eqnarray*} This completes the proof.
\end{proof}
\begin{Lemma}\label{clp} Let $ F(u)= (V_{\gamma} \ast |u|^2)u, p\in (2,3),$ $\frac{1}{p}+\frac{\gamma}{d}-1= \frac{1}{2p'}$ and pair $(p,r)$ is  Klein-Gordon admissible. Then   we have 
\begin{multline*} 
 \left \|  \int_{0}^{t} K(t- \tau)[ F(u(\tau)) - F(v(\tau))]  d\tau \right\|_{L^{r}_t (\R, M^{p,1}_s)}\\ \lesssim  \left(\|u\|_{L^{r}_t(\R, M^{p,1}_s)}^{2}+\|u\|_{L^{r}_t(\R, M^{p,1}_s)}\|v\|_{L^{r}_t(\R, M^{p,1}_s)}+ \|v\|_{L^{r}_t(\R, M^{p,1}_s}^{2}\right)  \|u-v\|_{L^{r}_t(\R, M^{p,1}_s)}^3.
\end{multline*}
\end{Lemma}
\begin{proof}
By Proposition \ref{nel}, we have 
\begin{eqnarray*}
\left\| \int_{0}^{t} K(t- \tau)[ F(u(\tau)) - F(v(\tau))]  d\tau \right\|_{L^{r}_t (\R, M^{p,1}_s)} \lesssim   \| F(u) - F(v))\|_{L^{r/3}_t (\R, M^{p',1}_s)}.
\end{eqnarray*}
By Proposition \ref{gap}, Lemma  \ref{rl}\eqref{ir}  and H\"older inequality, we obtain
\[ \| (V_{\gamma}\ast |u|^2)(u-v)\|_{L^{r/3} (\R, M^{p',1}_s)} \lesssim \|u\|^2_{L^{r} (\R, M^{p,1}_s)} \|u-v\|_{L^{r} (\R, M^{p,1}_s)}  \]
and
\begin{multline*} \| (V_{\gamma}\ast (|u|^2 - |v|^2))v\|_{L^{r/3} (\R, M^{p',1}_s)}\\
 \lesssim  \left( \|u\|_{L^{r} (\R, M^{p,1}_s)} \|v\|_{L^{r} (\R, M^{p,1}_s)} + \|v\|^2_{L^{r} (\R, M^{p,1}_s)} \right)  \|u-v\|_{L^{r} (\R, M^{p,1}_s)} .
\end{multline*}
\end{proof}
\begin{Lemma} [\cite{Bhimani2018global}]\label{iml} Let  $V$ be given by \eqref{hk},  $1\leq p \leq 2, 1\leq q < \frac{2d}{d+\gamma}$. Then for any $ f,g \in M^{p,q}(\mathbb R^{d})$, we have
\begin{enumerate}
\item $\|(V\ast |f|^{2}) f\|_{M^{p,q}} \lesssim\|f\|_{M^{p,q}}^{3}.$
\item $\| (V\ast |f|^{2})f - (K\ast |g|^{2})g\|_{M^{p,q}} \lesssim  (\|f\|_{M^{p,q}}^{2}+\|f\|_{M^{p,q}}\|g\|_{M^{p,q}}+ \|g\|_{M^{p,q}}^{2}) \|f-g\|_{M^{p,q}}.$
\end{enumerate}
\end{Lemma}
\section{Proofs of  Theorems \ref{mt} and \ref{LW}}\label{ps1}
\begin{proof}[Proof of Theorem \ref{mt}]  
Recall that equation \eqref{nlkgh} have the following
equivalent form
\[u(t)= K'(t)u_0 + K(t)u_1 -\int_{0}^{t} K(t-\tau) F(u(\tau)) d\tau =: \mathcal{J}(u)\]
where 
\[K(t)=\frac{\sin t (I-\Delta)^{1/2}}{(I-\Delta)^{1/2}}, \  \  K'(t)= \cos t (I-\Delta)^{1/2}, \ \  F(u)= (V_{\gamma} \ast |u|^2)u.\]
Denote $X= L^{r}(\R, M_s^{p,1} (\R^d)).$ 
For $\delta>0,$ put 
$ B_{\delta}=\{ u \in X: \|u\|_{X} \leq \delta \} -$ which is the closed ball of radius $\delta$, and centered at the origin in $X.$ 
Since  $rd\left( \frac{1}{2}- \frac{1}{p} \right) >1,$
we have $(1+ |t|)^{-d \left(\frac{1}{2}- \frac{1}{p} \right)} \in L^{r}(\R).$
Now by Proposition \ref{wp}, we have
\[ \|K(t) u_0 \|_{X} \lesssim  \left\| (1+ |t|)^{-d \left(\frac{1}{2}- \frac{1}{p} \right)}  \|u_0\|_{M^{p',1}_{s+2\sigma(p)}}\right\|_{L^{r}}  \lesssim   \|u_0\|_{M^{p',1}_{s+2\sigma(p)}}. \] 
By Propositions \ref{wp} and \ref{iso},  we have
\begin{eqnarray*}
 \|K'(t) u_1 \|_{X}  & \lesssim  &  \left\| (1+ |t|)^{-d\left(\frac{1}{2}- \frac{1}{p} \right)}  \|u_1\|_{M^{p',1}_{s+2\sigma(p)-1}}\right\|_{L^{r}} \lesssim   \|u_1\|_{M^{p',1}_{s+2\sigma(p)-1}}.
\end{eqnarray*}
By Proposition \ref{nel}, we have 
\[ \left \|  \int_{0}^{t} K(t-\tau)) F(u(\tau))  d\tau \right\|_{X} \lesssim \|u\|_{X}^3. \]
Thus we have 
\begin{eqnarray*}
\| \mathcal{J}(u)\|_{X} \lesssim  \|u_0\|_{M^{p',1}_{s+2\sigma(p)}}  + \|u_1\|_{M^{p',1}_{s+2\sigma(p)-1}} + \|u\|_{X}^3.
\end{eqnarray*}
By Lemma \ref{clp}, for any $u, v \in B_{\delta},$ we have 
\[ \|\mathcal{J}u - \mathcal{J}v \|_{X} \lesssim  \left(\|u\|^2_X+\|u\|_X\|v\|_X+ \|v\|_X^{2}\right) \|u-v\|_{X} .\]
If we assume that $\delta>0$  is sufficiently small, then $\mathcal{J}: X \to X$  is a strict contraction. Therefor $\mathcal{J}$ has a unique fixed point and we have  $u \in  L^{r}(\R, M^{p,1}_s (\R^d)).$   We shall now verify  this   $u \in C(\R, M^{p,1}_s(\R^d))\cap C^{1}(\R, M^{p,1}_{s-1}(\R^d))$ and  $\|u\|_{ L^{r}(\R, M_s^{p,1}(\R^d))} \lesssim \|u_0\|_{M^{p',1}_{s+ 2\sigma(p)}} + \|u_1\|_{ M^{p',1}_{s+2 \sigma(p)-1}}.$ To prove $u \in C(\R,M^{p,1}_s(\R^d)).$ It is equivalent to prove that 
\begin{eqnarray}\label{ss}
\|u(t_n, \cdot) - u(t, \cdot)\|_{M^{p,1}_s} \to 0 
\end{eqnarray}
 as  $ t_n \to t$ for arbitrary fixed $t>0.$
 We note that
\begin{eqnarray*}
\|u(t_n, \cdot)-u(t, \cdot)\|_{M^{p,1}_s}  & \leq &   \|K'(t_n)u_0-K'(t)u_0\|_{M^{p,1}_s} +\|K(t_n)u_1-K(t)u_1\|_{M^{p,1}_s}\\
&& +\left\|  \int_0^{t_n} K(t_n- \tau) F(u(\tau))-\int_0^{t} K(t- \tau)F(u(\tau)) \right\|_{M^{p,1}_s}\\
& = & I + II + III.
\end{eqnarray*}
Recall that $u_0, J^{-1}u_1\in M^{p,1}_s(\R^d)$ (see Proposition \ref{iso}). For $I, II,$ by density Lemma \ref{dl}, Proposition \ref{wp}, triangle inequality, and since $G(t)=e^{it\omega^{1/2}} \  (\omega = I-\Delta),$ we only need to  prove that $G(t)v\in C(\R, M^{p,1}_s(\R^d))$ for $v\in \mathcal{S}^{\Omega}.$ By Hausdroff-Young inequality, we have 
\begin{eqnarray*}
\|\square_k \left( G(t_n)v- G(t)v\right)\|_{L^p}  & \lesssim &  \left\| \sigma_k \left( e^{it_n (1+ |\xi|^2)^{1/2}}-e^{it (1+ |\xi|^2)^{1/2}}  \right) \hat{v}(\xi)\right\|_{L^{p'}}\\
& \lesssim  & \left\| \left( e^{it_n (1+ |\xi|^2)^{1/2}}-e^{it (1+ |\xi|^2)^{1/2}}  \right) \hat{v}(\xi)\right\|_{L^{p'}} \to 0
\end{eqnarray*}
as $t_n \to t,$  by Lebesgue dominated convergence theorem. Since $\hat{v}\in \mathcal{S}^{\Omega},$  there exists only finite number of $k$ such that $\square_k \left( G(t_n)v-G(t)v\right) \neq 0,$ so we have 
$\|G(t_n)v-G(t)v\|_{M^{p,1}_s} \to 0$ as $t_n\to t.$  It follows that  $I$ and $II$ tends to  0 as $t_n\to t.$
For $III,$ we note that
\begin{eqnarray*}
III & \lesssim &  \left\|  \int_0^{t_n} K(t_n- \tau) F(u(\tau)) d\tau -\int_0^{t_n} K(t- \tau)F(u(\tau)) d\tau  \right\|_{M^{p,1}_s}\\
&& +  \left\|  \int_0^{t_n} K(t- \tau) F(u(\tau)) d\tau -\int_0^{t} K(t- \tau)F(u(\tau)) d\tau  \right\|_{M^{p,1}_s}\\
& \lesssim &      \int_0^{t_n} \left\|  ( K(t_n- \tau) - K(t-\tau ) )F(u(\tau))  \right\|_{M^{p,1}_s} d\tau\\
&& +    \int_{t_n}^{t}  \left\|  K(t-\tau) F(u(\tau))  \right\|_{M^{p,1}_s} d\tau  \\
&= & \tilde{I}+\tilde{II}.
\end{eqnarray*}
For $ \left\| ( K(t_n- \tau) - K(t- \tau))F(u(\tau) \right\|_{M^{p,1}_s}\lesssim \|F(u(\tau))\|_{M^{p',1}_s}\lesssim \|u\|_{M^{p,1}_s}^3 \in L^r(\R).$  Since $ 3 \leq r,$ we have  $L^{r}[0, t] \subset L^1[0,t]$ and so $\|u\|_{M^{p,1}_s}^3 \in L^1[0,t]$, hence \[\left\| ( K(t_n- \tau) - K(t- \tau))F(u(\tau) \right\|_{M^{p,1}_s} \in L^1[0,t].\]
Since  $ \left\| ( K(t_n- \tau) - K(t- \tau))F(u(\tau) \right\|_{M^{p,1}_s} \to 0$ as $t_n\to t,$ therefore we have $\tilde{I} \to 0$ as $t_n\to t.$ Secondly as in the proof of Proposition \ref{nel}, we obtain
\begin{eqnarray*}
\tilde{II} &  \lesssim &   \int_{t_n}^{t} (1+ |t-\tau|)^{-d (1/2-1/p)} \|F(u(\tau))\|_{M^{p',1}_s} d\tau \\
& \lesssim &   \int_{t_n}^{t}  \|F(u(\tau))\|_{M^{p',1}_s} d\tau  \lesssim   \int_{t_n}^{t}  \|u\|_{M^{p,1}_s}^3 d\tau \to 0 \\
\end{eqnarray*}
as $t_n \to t$ as $ \|u\|_{M^{p,1}_s}^3\in L^1([0,t]).$  It follows that \eqref{ss} holds.

We now prove that $u_t(t)$ exists and is continuous in $M^{p,1}_s$ sense. 
For $u_0, J^{-1}u_1\in M^{p,1}_s(\R^d)$ (see Proposition \ref{iso}), and since $G(t)=e^{it\omega^{1/2}} \  (\omega = I-\Delta),$ we should only deal with the derivative of $G(t)\psi(x)$ for $\psi \in M^{p,1}_s(\R^d)$ and $\int_{0}^{t} K(t-\tau) F(u(\tau)) d\tau.$
By  Lemma \ref{dl}, for every $\epsilon>0,$ there exists $v\in \mathcal{S}^{\Omega} \cap M^{p,1}_s(\R^d)$ such that  $\|\psi - v\|_{M^{p,1}_s} < \epsilon.$ For the derivative of $G(t)\psi(x)$ at $t=t_3$ for $\psi \in M^{p,1}_s(\R^d),$ we have 
\begin{eqnarray*}
\left\|  \frac{ G(t)\psi - G(t_3) \psi}{t-t_3} - i \omega^{1/2} G(t_3) \psi \right\|_{M^{p,1}_{s-1}} & = & \left\|  \frac{ G(t)\psi - G(t_3) \psi}{(t-t_3) \omega^{1/2}} - i  G(t_3) \psi \right\|_{M^{p,1}_{s}}\\
& \leq & \left\|  \frac{ G(t)(\psi-v) - G(t_3) (\psi-v)}{(t-t_3) \omega^{1/2}} \right\|_{M^{p,1}_{s}}\\
&& + \left\|  \frac{ G(t)(v) - G(t_3) (v)}{(t-t_3) \omega^{1/2}} - i G(t_3) v \right\|_{M^{p,1}_{s}}\\
&& + \|iG(t_3)(\psi -v)\|_{M^{p,1}_s}\\
& = & IV + V+ VI.
\end{eqnarray*} 
For $V,$ by the Hausdroff-Young inequality and the Lebesgue dominated convergence theorem, we have 
\begin{eqnarray*}
\left\|  \square_k \left( \frac{ G(t)(v) - G(t_3) (v)}{(t-t_3) \omega^{1/2}} - i G(t_3) v \right) \right\|_{L^p}  & \lesssim &  \left\|  \sigma_k \left( \frac{ e^{it\langle \xi \rangle} - e^{it_3\langle \xi \rangle}}{(t-t_3)  \langle \xi \rangle} - i e^{it_3 \langle \xi \rangle} \right) \hat{v} \right\|_{L^{p'}}\\
& \lesssim & \to 0 \ \text{as} \ t\to t_3. 
\end{eqnarray*}
As  $v\in \mathcal{S}^{\Omega} \cap M^{p,1}_s(\R^d),$ so there is only the finite number of $k$ such that  \[   \left( \frac{ G(t)(v) - G(t_3) (v)}{(t-t_3) \omega^{1/2}} - i G(t_3) v \right) \neq 0.\] Thus we get $V\to 0$ as $t\to t_3,$ that is, $(G(t) v(x))_t=i\omega^{1/2} G(t) v(x)$ in $M^{p,1}_{s-1}(\R^d)$  for $v\in \mathcal{S}^{\Omega} \cap M^{p,1}_s(\R^d).$
For $IV,$ by the Bernstein multiplier theorem, we have 
\begin{eqnarray*}
\left\|  \square_l \left( \frac{ G(t)(\psi -v) - G(t_3) (\psi- v)}{(t-t_3) \omega^{1/2}}  \right) \right\|_{L^p}\lesssim \|\psi -v \|_{L^p}.
\end{eqnarray*}
Using the almost orthogonality  of modulation space, we have $IV \lesssim \|\psi-v \|_{M^{p,1}_s} < \epsilon.$  For $VI,$ by Proposition \ref{wp} \eqref{kb}, we have $VI=\|iG(t_3)(\psi -v)\|_{M^{p,1}_s}\lesssim \|\psi -v \|_{M^{p,1}_s} < \epsilon.$
Accordingly, for $\psi \in M^{p,1}_s(\R^d),$ 
\begin{eqnarray}\label{dim}
(G(t) \psi)_t = i\omega^{1/2} G(t)\psi \ \text{in} \ M^{p,1}_s(\R^d).
\end{eqnarray}
For the nonlinear part,
\begin{equation*}
\begin{multlined}
\left\|   \frac{\int_{0}^{t} K(t-\tau) F(u(\tau)) d\tau- \int_{0}^{t_3} K(t_3-\tau) F(u(\tau)) d\tau}{t-t_3} -\int_{0}^{t_3} K'(t_3-\tau) F(u) d\tau\right\|_{M^{p,1}_{s-1}}\\ 
\leq \left\|   \frac{\int_{0}^{t_3} ( K(t-\tau) - K(t_3-\tau) F(u (\tau)) d\tau}{t-t_3} -\int_{0}^{t_3} K'(t_3-\tau) F(u) d\tau\right\|_{M^{p,1}_{s-1}}\\
+ \left\|   \frac{\int_{t_3}^{t}  K(t-\tau)  F(u (\tau)) d\tau}{t-t_3} \right\|_{M^{p,1}_{s-1}}\\
\lesssim \int_{0}^{t_3} \left\|  \left( \frac{ ( K(t-\tau) - K(t_3-\tau) }{t-t_3} - K'(t_3-\tau) \right) F(u) \right\|_{M^{p,1}_{s-1}} d\tau\\
 +  \max_{\tau \in   [t_3, t]} \|K(t-\tau) F(u(\tau))\|_{M^{p,1}_{s-1}}.
\end{multlined}
\end{equation*}
If $\omega(t,x) \in C(I, M^{p,1}_s(\R^d)),$ then we have $K(t) \omega(t,x) \in C(I, M^{p,1}_{s-1}(\R^d)).$ In fact taking taking advantage of \eqref{dim} and the Lebesgue dominated convergence theorem, we can get 
\begin{eqnarray*}
\|K(t)\omega (t,x) -K(t_3)\omega(t_3,x)\|_{M^{p,1}_{s-1}} & \leq & \|(K(t) -K(t_3))\omega(t_3,x)\|_{M^{p,1}_{s-1}} \\
&& + \|K(t) (\omega (t,x) -\omega(t_3,x))\|_{M^{p,1}_{s-1}}\\
&  \ & \to 0  \ \text{as} \ t\to t_3.
\end{eqnarray*}
Recall that $F(u) \in C(\R, M^{p,1}_s(\R^d))$ and apply \eqref{dim} and  the Lebesgue dominated convergence theorem, we can get \[\left. \left( \int_{0}^{t} K(t-\tau) F(u(\tau)) d\tau\right)'_{t}  \right\vert_{t=t_3} =  \int_{t=0}^{t_3} K'(t_3-\tau) F(u(\tau)) d\tau   \  \text{in} \   M^{p,1}_{s-1}(\R^d).  \] 
Consequently,
\begin{eqnarray*}
u_t(t)= -J^2K(t)u_0 + K'(t) u_1 - \int_{0}^{t} K'(t-\tau) F(u(\tau)) d\tau \ \text{in} \ M^{p,1}_s(\R^d).
\end{eqnarray*}
Next, the proof of time continuity of $u_t$ is similar to  $u$.
It only needs to take care of the difference of smoothness
and the action of the Bessel potential. Finally, we obtain  $u \in C(\R, M^{p,1}_s(\R^d)) \cap C^{1}(\R, M^{p,1}_{s-1}(\R^d)).$ 
\end{proof}
\begin{proof}[Proof of Corollary \ref{mts}]
Let 
\[ 2v_1(t)=  u_0 + \frac{u_1}{i  \omega^{1/2}}- \int_0^t \frac{G(-\tau) F(u(\tau))}{i\omega^{1/2}} d\tau\]
and 
\[ 2v_2(t)=  u_0 - \frac{u_1}{i  \omega^{1/2}}+ \int_0^t \frac{G(-\tau) F(u(\tau))}{i\omega^{1/2}} d\tau.\]
For $0<s<t,$ we  have 
\[ v_1(t)-v_1(s) = -\int_s^t \frac{G(-\tau) F(u(\tau))}{i\omega^{1/2}} d\tau.\]
Since pair $(p,r)$ is  Klein-Gordon admissible, we  there exists $\tilde{\beta}$ such that
\[ \frac{1}{\tilde{\beta}} + \frac{3}{r}=1, \   \   \tilde{\beta}d \left(\frac{1}{2}- \frac{1}{p} \right)>1. \]

By Proposition \ref{nel} and H\"older inequality, we have 
\begin{eqnarray*}
 \|v_1(t)-v_1(s)\|_{M^{p,1}_s} & \lesssim & \int_s^t (1+ |\tau|) ^{-d (\frac{1}{2} -\frac{1}{p})} \|F(u(\tau)) \|_{M^{p',1}_s}  d\tau \\
 & \lesssim &   \int_s^t (1+ |\tau|) ^{-d (\frac{1}{2} -\frac{1}{p})} \|u \|^3_{M^{p,1}_s}  d\tau \\
 & \lesssim &  \|(1+ |\tau|) ^{-d (\frac{1}{2} -\frac{1}{p})}\|_{L^{\tilde{\beta}}} \| \|u \|^3_{M^{p,1}_s} \|_{L^{r/3}([s,t], M^{p,1}_s)} \\
 & \lesssim &  \|u\|_{L^r([s,t], M^{p,1}_s)}^3.
\end{eqnarray*}
Since $\|u\|_{L^r([s,t], M^{p,1}_s)} \leq M $, we have 
 \[\|v_1(t)-v_1(s)\|_{M^{p,1}_s}  \lesssim   \|u\|_{L^r([s,t], M^{p,1}_s)}^3 \to 0 \  \text {as} \ t,s \to \infty.\]
 This implies that $v_1(t)$ is Cauchy in $M^{p,1}_s(\R^d)$ as $t\to  \infty.$ Denote $v_1^{+}$ to be the limit:
 \begin{eqnarray*}
 2v_1^{+} & = &  \lim_{t\to + \infty} 2v_1(t)\\
 & = & u_0 +  \frac{u_1}{i  \omega^{1/2}}- \int_0^t \frac{G(-\tau) F(u(\tau))}{i\omega^{1/2}} d\tau
 \end{eqnarray*}
and 
\begin{eqnarray*}
 2v_1^{-} & = &  \lim_{t\to + \infty} 2v_1(t)\\
 & = & u_0 -  \frac{u_1}{i  \omega^{1/2}}+ \int_0^t \frac{G(-\tau) F(u(\tau))}{i\omega^{1/2}} d\tau.
 \end{eqnarray*}
Similarly, we obtain
\[ v_2^{+}(t)= \lim_{t\to \infty} v_2(t) \  \ \text{and} \ v_2^{-}(t)= \lim_{t\to \infty} v_2(t). \]
Recall that $v^{\pm}= G(t)v_{1}^{\pm} + G(t)v_2^{\pm},$ we note that
\begin{eqnarray*}
\|u(t)-v^{+}\|_{M^{p,1}_s} &  = &  \left \|  \int_{t}^{\infty} K(t- \tau) F(u(\tau))  d\tau \right\|_{M^{p,1}_s}\\
& \lesssim &  \|(1+ |\tau|) ^{-d (\frac{1}{2} -\frac{1}{p})}\|_{L^{\tilde{\beta}}} \| \|u \|^3_{M^{p,1}_s} \|_{L^{r/3}([t,\infty], M^{p,1}_s)}\\
& \lesssim &    \|u\|_{L^3([t,\infty], M^{p,1}_s)}^3 \to 0 \  \text {as} \ t \to \infty.
\end{eqnarray*}
So  is $v^{-}$ respectively. In fact, in our proof  we also have $v_1^{+} \in M^{p,1}_s(\R^d).$
\end{proof}
\begin{proof}[Proof of Theorem \ref{LW}]
Equation \eqref{nlw} can be written in the equivalent form
\begin{equation}
u(\cdot, t)= \tilde{K}(t)u_{0} +K(t)u_{1}- \int_{0}^{t} K(t-\tau)[ (V_{\gamma}\ast |u|^2)(\tau)u (\tau)] d\tau=:\mathcal{J}(u)
\end{equation}
where
$$
K(t)=\frac{\sin (t\sqrt{-\bigtriangleup})}{\sqrt{-\bigtriangleup}}, \tilde{K} (t) =\cos (t\sqrt{-\bigtriangleup} ).
$$
By using Proposition \ref{wpm} for the first two inequalities below, and Propositions \ref{fip}  and  \ref{iml} for the last inequality, we can write
\begin{equation}
\begin{cases}
\|\tilde{K}(t)u_{0}\|_{X} \leq C_{T} \|u_{0}\|_{X},\\
\|K(t)u_{1}\|_{X} \leq C_{T} \|u_{1}\|_{X},\\
\|\int_{0}^{t} K(t-\tau)[ (V_{\gamma}\ast |u|^2)(\tau)u (\tau)] d\tau \|_{X} \leq  T C_{T} \|u\|_{X}^3,
\end {cases}
\end{equation}
where $C_T$ is some constant times $(1+T^2)^{d/4}$, as before. Thus the standard contraction mapping argument can be applied to $\mathcal{J}$ to complete the proof. This completes the proof of Theorem \ref{LW} \eqref{lw2}. Taking Propositions \ref{wp}, \ref{iml} and  Corollary \ref{afi} and Lemma \ref{cl} into account,  the    standard contraction mapping argument give the proof of Theorem \ref{LW} \eqref{lw1}.
\end{proof}
\section{Proofs of Theorems \ref{sdhte}, \ref{scft} and \ref{lwhte}}\label{ps2}
In order to prove Theorem \ref{scft}  first we shall prove following Strichartz type estimates for Schr\"odinger admissible pairs.  Specifically, we have 
\begin{Proposition}\label{snel}   Let $ F(u)= (V_{\gamma}\ast |u|^2)u,  p\in (2,3),$ $\frac{1}{p}+\frac{\gamma}{d}-1= \frac{1}{2p'}$ and   pair $(p,r)$ is  Schr\"odinger admissible. Then we have
\begin{eqnarray*}
\left \|  \int_{0}^{t} S(t-\tau) F(u(\tau))  d\tau \right\|_{L^{r}_t (\R, M^{p,1}_s)} &  \lesssim  &\|F(u)\|_{L^{r/3}_t(\R, M^{p',1}_s)}  \lesssim  \|u\|_{L^{r}_t(\R, M^{p,1}_s)}^3.
\end{eqnarray*}
\end{Proposition}
\begin{proof} By general Minkowski inequality, Proposition \ref{uf},  we have 
\begin{eqnarray*}
\left \|  \int_{0}^{t} S(t-\tau) F(u(\tau))  d\tau \right\|_{L^{r}_t (\R, M^{p,1}_s)} & \lesssim & \left \|  \int_{0}^{t} \|S(t-\tau) F(u(\tau)) \|_{M^{p,1}_s} d\tau \right\|_{L^{r}_t(\R)}\\
& \lesssim & \left \|  \int_{0}^{t} (1+ |t-\tau|)^{-d(1/2-1/p)} \|F(u)\|_{M^{p',1}_{s}} d\tau \right\|_{L^{r}_t(\R)}\\
& \lesssim & \left \|  \int_{\mathbb R } (1+ |t-\tau|)^{-d(1/2-1/p)} h(\tau) d\tau\right\|_{L^{r}_t(\mathbb R)}\\
& \lesssim &  \| g\ast h \|_{L^r_t},
\end{eqnarray*}
where  $ h(\tau)=  \|F(u)\|_{M^{p',1}_{s}}, g(t)=(1+ |t|)^{-d(1/2-1/p)}. $
We divide Schro\"odinger admissible pairs (see \eqref{sra}) into several cases.
\item\textbf{ Case I}:  $\frac{1}{\beta}<d \left( \frac{1}{2} - \frac{1}{p} \right)\wedge 1.$
In this case  we have 
\[ d\beta \left( \frac{1}{2} - \frac{1}{p} \right)>1.\]
Using Young inequality and H\"older's inequality we have
\begin{eqnarray*}
\left \|  \int_{0}^{t} S(t-\tau) F(u(\tau))  d\tau \right\|_{L^{r}_t (\R, M^{p,1}_s)} & \lesssim & \|g\|_{L^{\beta}} \|F(u)\|_{L^{r/3}(\R, M_s^{p',1}(\R^d))}\\
& \lesssim & \|u\|_{L^{r}(\R, M_s^{p,1}(\R^d))}^3.
\end{eqnarray*}
\item \textbf{Case II}: $\frac{1}{\beta}= 1\wedge d \left( \frac{1}{2} - \frac{1}{p} \right), d \left( \frac{1}{2} - \frac{1}{p} \right)>1.$
In this case, we can get  $\beta=1$ and $r= \infty.$ Obviously
\[ d\beta \left( \frac{1}{2} - \frac{1}{p} \right)>1,\]
and therefor, we have the desired result by the same way as Case I. 
\item\textbf{ Case III}: $\frac{1}{\beta}= 1\wedge d \left( \frac{1}{2} - \frac{1}{p} \right), d \left( \frac{1}{2} - \frac{1}{p} \right)<1.$
In this case we have
\[ d\beta \left( \frac{1}{2} - \frac{1}{p} \right)=1.\]
Since pair $(p,r)$ is  Schr\"odinger  admissible, we have
\[\frac{1}{r}= \frac{3}{r}- \frac{1- d(1/2-1/p)}{1}\]
and $r/3>1.$ By Hardy-Littlewood-Sobolev inequality in one dimension, we have 
\begin{eqnarray*}
\left \|  \int_{0}^{t} K(t-\tau) F(u(\tau))  d\tau \right\|_{L^{r}_t (\R, M^{p,1}_s)} 
& \lesssim &  \| g\ast h \|_{L^r_t (\mathbb R)}\\
& \lesssim  &  \left\|\|F(u)\|_{M^{p',1}_{s}}\right\|_{L^{r/3}}\\
& \lesssim & \|F(u)\|_{L^{r/3}(\R, M^{p',1}_s)} \lesssim \|u\|_{L^{r}(\R, M^{p,1}_s)}^3.
\end{eqnarray*}
\item \textbf{Case IV}: $\frac{1}{\beta}= 1\wedge d \left( \frac{1}{2} - \frac{1}{p} \right), d \left( \frac{1}{2} - \frac{1}{p} \right)=1.$
In  this case
\[ (p,r) = \left( \frac{2d}{d-2}, \infty \right)\]
which is not Schr\"odinger admissible.
\end{proof}

\begin{Lemma}\label{csap} Let $ F(u)= (V_{\gamma} \ast |u|^2)u, p\in (2,3),$ $\frac{1}{p}+\frac{\gamma}{d}-1= \frac{1}{2p'}$ and pair $(p,r)$ is  Schr\"odinger  admissible. Then   we have 
\begin{multline*} 
 \left \|  \int_{0}^{t} S(t- \tau)[ F(u(\tau)) - F(v(\tau))]  d\tau \right\|_{L^{r}_t (\R, M^{p,1}_s)}\\ \lesssim  \left(\|u\|_{L^{r}_t(\R, M^{p,1}_s)}^{2}+\|u\|_{L^{r}_t(\R, M^{p,1}_s)}\|v\|_{L^{r}_t(\R, M^{p,1}_s)}+ \|v\|_{L^{r}_t(\R, M^{p,1}_s}^{2}\right)  \|u-v\|_{L^{r}_t(\R, M^{p,1}_s)}^3.
\end{multline*}
\end{Lemma}
\begin{proof}
Using Propositions \ref{gap} and \ref{snel}, Lemma \ref{rl} \eqref{ir} and H\"older inequality, the proof can be produced.  We omit the details. 
\end{proof}

\begin{proof}[Proof of Theorem \ref{sdhte}] For $\alpha=2,$ we may rewrite equation \eqref{fHTE} in the following form
\[u(t)=S(t)u_0  - \int_{0}^{t} S(t-\tau) F(u(\tau)) d\tau=:\mathcal{J}(u)\]
where $S(t)=e^{-it\Delta}$ and $F(u)=(V_{\gamma}\ast |u|^2)u.$
Denote $X= L^{r}(\R, M_s^{p,1} (\R^d)).$ 
For $\delta>0,$ we put 
$ B_{\delta}=\{ u \in X: \|u\|_{X} \leq \delta \} -$ which is the closed ball of radius $\delta$, and centered at the origin in $X.$
Since $ rd\left( \frac{1}{2}- \frac{1}{p} \right) >1,$ we have
$(1+ |t|)^{-d \left(\frac{1}{2}- \frac{1}{p} \right)} \in L^{r}(\R).$ Now
by  Proposition \ref{uf}, we have
\[ \|S(t) u_0 \|_{X} \lesssim  \left\| (1+ |t|)^{-d \left(\frac{1}{2}- \frac{1}{p} \right)}  \|u_0\|_{M^{p',1}_{s}}\right\|_{L^{r}}   \lesssim   \|u_0\|_{M^{p',1}_{s}}. \] 
By Proposition \ref{snel}, we have 
\[ \left \|  \int_{0}^{t} S(t-\tau)) F(u(\tau))  d\tau \right\|_{X} \lesssim \|u\|_{X}^3. \]
Thus we have 
\begin{eqnarray*}
\| \mathcal{J}(u)\|_{X} \lesssim  \|u_0\|_{M^{p',1}_{s}}   + \|u\|_{X}^3.
\end{eqnarray*}
By Lemma \ref{csap}, for any $u, v \in B_{\delta},$ we have 
\[ \|\mathcal{J}u - \mathcal{J}v \|_{X} \lesssim  \left(\|u\|^2_X+\|u\|_X\|v\|_X+ \|v\|_X^{2}\right) \|u-v\|_{X} .\]
If we assume that $\delta>0$  is sufficiently small, then $\mathcal{J}: X \to X$  is a strict contraction. Therefor $\mathcal{J}$ has a unique fixed point and we have $u\in L^{r}(\R, M^{p,1}_s(\R^d))$ and  $\|u\|_{ L^{r}(\R, M_s^{p,1}(\R^d))} \lesssim \|u_0\|_{M^{p',1}_{s}}.$
We want to show that if  $f\in M^{p,1}_s(\R^d)$ then  $S(t)f \in C(\R, M^{p,1}_s(\R^d)).$
Let $t>0$ and $t_n \to t.$  By Lemma \ref{dl}, Proposition \ref{uf} and  the triangle inequality, we have 
 \begin{eqnarray*}
 \|S(t)f- S(t_n)f\|_{M^{p,1}_s}  & \leq&  \|S(t)f-S(t)g\|_{M^{p,1}_s}+ \|S(t)g-S(t_n)g\|_{M^{p,1}_s}\\
 &&+\|S(t_n)f-S(t_n)g\|_{M^{p,1}_s}.
 \end{eqnarray*}
 We only need to treat  the case $f\in \mathcal{S}^{\Omega}.$ Using Lemma \ref{lpm} and the Hausdroff-Young inequality, we have 
 \begin{eqnarray*}
 \| \square_k (S(t_n)-S(t))f \|_{L^p} \lesssim \| (S(t_n)-S(t))f \|_{L^{p}} \lesssim  \| (e^{it_n |\xi|^2}-e^{it|\xi|^2}) \hat{f}\|_{L^{p'}}\to 0
 \end{eqnarray*}
 as $ t_n\to t$
 by Lebesgue dominated convergence theorem. Since $f\in \mathcal{S}^{\Omega},$ there exist only finite number of $k$ such that 
 \[  \square_k (S(t_n)-S(t))f \neq 0\]
 and thus we have 
 \[ \|S(t)f- S(t_n)f\|_{M^{p,1}_s} \to 0 \ \text{as} \ t_n \to t.\]
  We write
 \begin{eqnarray*}
 I & = &   \int_0^{t} S(t-\tau) F(u(\tau)) d\tau- \int_0^{t_n}S(t_n-\tau) F(u(\tau)) d\tau \\
&  = &  \left( \int_0^{t_n}S(t-\tau) F(u(\tau)) d\tau - \int_0^{t_n}S(t_n-\tau) F(u(\tau)) d\tau\right) \\
&& + \left( \int_0^{t}S(t-\tau) F(u(\tau)) d\tau-\int_0^{t_n}S(t-\tau) F(u(\tau)) d\tau \right) \\
& = & I_1+ I_2.
 \end{eqnarray*} 
For $I_2,$ we have 
\begin{eqnarray*}
\|I_2\|_{M^{p,1}_s} & \lesssim &  \int_{t_n}^{t} \|S(t- \tau) F(u)(\tau)\|_{M^{p,1}_s} d\tau \\
& \lesssim &  \int_{t_n}^{t} (1+ |t-\tau|)^{-d(1/2-1/p)} \|F(u)(\tau)\|_{M^{p',1}_s} d\tau \\
& \lesssim &  \int_{t_n}^{t}  \|u\|_{M^{p,1}_s}^3 d\tau\\
& \lesssim &  |t-t_n|^{\beta} \|u\|^3_{L^{r} ([0, t], M^{p,1}_s)}\to 0.
\end{eqnarray*} 
For $I_1,$ we have 
\begin{eqnarray*}
I_1 & \lesssim & \int_0^{t_n} \|S(\tau) (S(t_n)- S(t)) F(u(\tau))\|_{M^{p,1}_s} d\tau \\
& \lesssim & \int_{I} \|(S(t_n)-S(t)) F(u(\tau))\|_{M^{p,1}_s} d\tau. 
\end{eqnarray*}
We note that  $\|(S(t_n)-S(t)) F(u(\tau))\|_{M^{p,1}_s} \lesssim \|F(u)(\tau)\|_{M^{p,1}_s}^3 $ and recalling $r\geq 3$ and $u \in L^{r}(\R, M^{p,1}_s(\R^d)),$ we have  $\|u(\tau)\|_{M^{p,1}_s}^3\in L^{1}[0,t].$  Since $F(u) \in M^{p,1}_s(\R^d),$ for every $\tau \in [0, t]$ we have  $\|(S(t_n)-S(t)) F(u(\tau))\|_{M^{p,1}_s} \to 0.$
\end{proof}
\begin{proof}[Proof of Corollary \ref{ssdhte}] We only prove the statement for $u_{+},$ since the proof for $u_{-}$ follows similarly. Let us first construct the scattering state $u_{+}(0).$ For $t>0$ define $v(t)=e^{-it\Delta}u(t).$  We will show that $v(t)$ converges in $M^{p,1}_s(\R^d)$ as $t\to \infty,$ and define $u_{+}$ to be the limit. Indeed from  Duhamel's formula we have 
\begin{eqnarray}\label{te}
v(t)=u_0  - \int_{0}^{t} e^{-i\tau \Delta} F(u(\tau)) d\tau \   \ (F(u)= (V_{\gamma} \ast |u|^2)u).
\end{eqnarray}
Therefore, for $0<s<t,$ we  have 
\[ v(t)-v(s) = -i\int_s^t e^{-i\tau \Delta} F(u(\tau)) d\tau.\]
Since pair $(p,r)$ is  a Schr\"odinger admissible,  there exists $\tilde{\beta}$ such that
\[ \frac{1}{\tilde{\beta}} + \frac{3}{r}=1, \   \   \tilde{\beta}d \left(\frac{1}{2}- \frac{1}{p} \right)>1. \]
By Proposition \ref{nel} and H\"older inequality, we have 
\begin{eqnarray*}
 \|v(t)-v(s)\|_{M^{p,1}_s} & \lesssim & \int_s^t (1+ |\tau|) ^{-d (\frac{1}{2} -\frac{1}{p})} \|F(u(\tau)) \|_{M^{p',1}_s}  d\tau \\
 & \lesssim &   \int_s^t (1+ |\tau|) ^{-d (\frac{1}{2} -\frac{1}{p})} \|u \|^3_{M^{p,1}_s}  d\tau \\
 & \lesssim &  \|(1+ |\tau|) ^{-d (\frac{1}{2} -\frac{1}{p})}\|_{L^{\tilde{\beta}}} \| \|u \|^3_{M^{p,1}_s} \|_{L^{r/3}([s,t], M^{p,1}_s)} \\
 & \lesssim &  \|u\|_{L^r([s,t], M^{p,1}_s)}^3.
\end{eqnarray*}
Since $\|u\|_{L^r(\R, M^{p,1}_s)} \leq M $, we have 
 \[\|v(t)-v(s)\|_{M^{p,1}_s}  \lesssim   \|u\|_{L^r([s,t], M^{p,1}_s)}^3 \to 0 \  \text {as} \ t,s \to \infty.\]
 This implies that $v(t)$ is Cauchy in $M^{p,1}_s(\R^d)$ as $t\to  \infty.$  We define $u_{+}$ to be the limit. In view of \eqref{te},  we see that
\[ u_+(0)= u_0  - \int_{0}^{\infty} e^{-i\tau \Delta} F(u(\tau)) d\tau\]
and thus 
 \[ u_+(t)= e^{it\Delta}u_0  - \int_{0}^{\infty} e^{i(t-\tau) \Delta} F(u(\tau)) d\tau.\]
 We note that 
 \begin{eqnarray*}
\|u(t)-e^{it\Delta}u_{+}\|_{M^{p,1}_s} &  = &  \left \|  \int_{t}^{\infty} S(t- \tau) F(u(\tau))  d\tau \right\|_{M^{p,1}_s}\\
& \lesssim &  \|(1+ |\tau|) ^{-d (\frac{1}{2} -\frac{1}{p})}\|_{L^{\tilde{\beta}}} \| \|u \|^3_{M^{p,1}_s} \|_{L^{r/3}([t,\infty], M^{p,1}_s)}\\
& \lesssim &    \|u\|_{L^r([t,\infty], M^{p,1}_s)}^3 \to 0 \  \text {as} \ t \to \infty.
\end{eqnarray*}
In fact, in our proof  we also have $ e^{it\Delta}u_0, e^{it\Delta}u_+ \in M^{p,1}_s(\R^d).$
\end{proof}
In order to prove Theorem \ref{scft} first we recall following
\begin{Lemma}[\cite{bhimani2018nonlinear}]\label{lcg} Let  $V\in M^{\infty,1}(\mathbb R^d),$  and $1\leq p, q \leq 2.$ For $f \in M^{p,q}(\R^d),$ we have $$ \|(V\ast |f|^{2}) f\|_{M^{p,q}} \lesssim\|f\|_{M^{p,q}}^{3} ,$$and
$$\| (V\ast |f|^{2})f - (V\ast |g|^{2})g\|_{M^{p,q}} \lesssim  (\|f\|_{M^{p,q}}^{2}+\|f\|_{M^{p,q}}\|g\|_{M^{p,q}}+ \|g\|_{M^{p,q}}^{2}) \|f-g\|_{M^{p,q}}.$$
\end{Lemma}
\begin{proof}[Proof of Theorem \ref{scft}]  Recall \eqref{fHTE} can be written in the equivalent form
\begin{equation*}\label{df11}
u(\cdot, t)= U(t)u_{0}-i\int_{0}^{t}U(t- \tau)[(V\ast |u|^2)u] \, d\tau=:\mathcal{J}(u).
\end{equation*}
We first  prove the local existence on $[0,T)$ for some $T>0.$  
By Minkowski's inequality for integrals,  Proposition \ref{uf} and  Lemma \ref{lcg}, we obtain
\begin{eqnarray*}
\left\| \int_{0}^{t} U(t-\tau) [(V\ast |u|^{2}(\tau)) u(\tau)]  \, d\tau \right\|_{M^{p,q}} 
   & \leq &  c T(1+|t|)^{d\left| \frac{1}{p}-\frac{1}{2} \right|} \|u(t)\|_{M^{p,p}}^{3},
\end{eqnarray*}
 for some universal constant $c.$
By Proposition \ref{uf} and the above inequality, we have
\begin{eqnarray*}
\|\mathcal{J}u\|_{C([0, T], M^{p,q})} \leq  C_{T} ( \|u_{0}\|_{M^{p,q}} + cT  \|u\|_{M^{p,q}}^{3})
\end{eqnarray*}
where $C_{T}=(1+|T|)^{d\left| \frac{1}{p}-\frac{1}{2} \right|}.$
For $M>0$, put  $$B_{T, M}= \{u\in C([0, T], M^{p,q}(\mathbb R^{d})):\|u\|_{C([0, T], M^{p,q})}\leq M \},$$  which is the  closed ball  of radius $M$, and centered at the origin in  $C([0, T], M^{p,q}(\mathbb R^{d}))$.  
Next, we show that the mapping $\mathcal{J}$ takes $B_{T, M}$ into itself for suitable choice of  $M$ and small $T>0$. Indeed, if we let, $M= 2C_{T}\|u_{0}\|_{M^{p,p}}$ and $u\in B_{T, M},$ it follows that 
\begin{eqnarray*}
\|\mathcal{J}u\|_{C([0, T], M^{p,p})} \leq  \frac{M}{2} + c C_{T}T M^{3}.
\end{eqnarray*}
We choose a  $T$  such that  $c C_{T}TM^{2} \leq 1/2,$ that is, $T \leq \tilde{T}(\|u_0\|_{M^{p,p}})$ and as a consequence  we have
\begin{eqnarray*}
\|\mathcal{J}u\|_{C([0, T], M^{p,p})} \leq \frac{M}{2} + \frac{M}{2}=M,
\end{eqnarray*}
that is, $\mathcal{J}u \in B_{T, M}.$
By Lemma \ref{lcg}, and the arguments as before, we obtain
\begin{eqnarray*}
\|\mathcal{J}u- \mathcal{J}v\|_{C([0, T], M^{p,q})} \leq \frac{1}{2} \|u-v\|_{C([0, T], M^{p,q})}.
\end{eqnarray*}
Therefore, using Banach's contraction mapping principle, we conclude that $\mathcal{J}$ has a fixed point in $B_{T, M}$ which is a solution of \eqref{fHTE}. 

Indeed, the solution constructed before is global in time: in view of the conservation of $L^{2}$ norm, Proposition \ref{gap} and Lemma \ref{rl}, we have
\begin{eqnarray*}
\|u((t)\|_{M^{p,p}} & \lesssim  &  C_{T} \left(  \|u_{0} \|_{M^{p,q}} + \int_{0}^{t} \|V\ast |u(\tau)|^{2}\|_{M^{\infty,1}} \|u(\tau)\|_{M^{p,q}} d\tau \right) \nonumber \\
& \lesssim &   C_{T} \left(  \|u_{0}\|_{M^{p,q}} + \int_{0}^{t} \|V\|_{M^{\infty,1}} \||u(t)|^{2}\|_{M^{1, \infty}} \|u(\tau)\|_{M^{p,q}} d\tau \right) \nonumber \\
& \lesssim &  C_{T} \left(  \|u_{0}\|_{M^{p,q}} + \int_{0}^{t}  \||u(t)|^{2}\|_{L^1} \|u(\tau)\|_{M^{p,q}} d\tau \right) \nonumber \\
& \lesssim &  C_{T} \left( \|u_{0}\|_{M^{p,p}} +  \|u_{0}\|_{L^{2}}^{2} \int_{0}^{t}\|u(\tau)\|_{M^{p,p}} d\tau \right)
\end{eqnarray*}
and by Gronwall inequality, we conclude that $\|u(t)\|_{M^{p,q}}$ remains bounded on finite time intervals. This completes the proof.
\end{proof}
\begin{proof}[Proof of Theorem \ref{lwhte}] Recall \eqref{fHTE} can be written in the equivalent form
\begin{equation*}\label{df11}
u(\cdot, t)= U(t)u_{0}-i\int_{0}^{t}U(t- \tau)[(V\ast |u|^2)u] \, d\tau=:\mathcal{J}(u).
\end{equation*}
By using Proposition \ref{uf} and  Corollary \ref{afi}, we can write
\begin{equation}
\begin{cases}
\|U(t)u_{0}\|_{M^{p,1}_s} \leq C_{T} \|u_{0}\|_{M^{p,1}_s},\\
\|\int_{0}^{t} U(t-\tau)[ (V_{\gamma}\ast |u|^2)(\tau)u (\tau)] d\tau \|_{X} \leq  T C_{T} \|u\|_{M_s^{p,1}}^3,
\end {cases}
\end{equation}
where $C_T$ is some constant times $(1+T^2)^{d/4}$, as before. Thus the standard contraction mapping argument can be applied to $\mathcal{J}$ to complete the proof.
\end{proof}
\section{Local well-posedness with potential  $V \in \mathcal{F}L^q$ or   $M^{1, \infty}$ or $M^{\infty, 1}$}\label{lwtp}
We consider generalized Klein-Gordon equation with  Hartree type linearity:
\begin{equation}\label{gnlkgh}
u_{tt}+ (I-\Delta)u = (V\ast |u|^{2k})u, 
u(0)=u_0, u_t(0)=u_1  \ (k\in \mathbb N).
\end{equation}
When $k=1,$ equation \eqref{gnlkgh} is coincides with \eqref{nlkgh}.
\begin{Theorem}[Local well-posedness]\label{wht1} Let $i=0,1.$
\begin{enumerate}
\item Assume that $V \in \mathcal{F}L^q (\R^d) (1<q< \infty)$  and
   $u_{i}\in M^{1,1}(\R^d) $ Then  there exists  
    $T^{*}=T^{\ast}(\|u_{i}\|_{M^{1,1} }) $ such that \eqref{gnlkgh} has a unique solution $u\in C([0, T^{*}), M^{1,1} (\mathbb R^d)).$ 
   \item Assume that $V \in \mathcal{F}L^q(\R^d)$  with $ 1<q<2,$ and  $u_{i}\in M^{p,\frac{2r}{2r-1}}(\mathbb R^{d}) \  ( 1 \leq p \leq 2, q<r< \infty).$   Then there exists
\ $T^{*}=T^{\ast}(\|u_{i}\|_{ M^{p,\frac{2r}{2r-1}}})$ such that \eqref{nlkgh} has a unique solution $u\in C([0, T^{*}), M^{p,\frac{2r}{2r-1}}(\mathbb R^{d})).$ 
\item  Assume that $V\in M^{\infty, 1}(\R^d) $  and $u_i \in M^{p,q}(\R^d).$ Then   there exists       $T^{*}=T^{\ast}(\|u_{i}\|_{M^{p,q}})$ such that \eqref{nlkgh} has a unique solution $u\in C([0, T^{*}), M^{p,q} (\mathbb R^d)).$  
   
   \item Assume that $V\in M^{\infty, 1}(\R^d) $  and $u_i \in M^{p,q}(\R^d) \  (1\leq p, q \leq 4, 1 \leq q \leq \frac{2^{2k-2}}{2^{2k-2}-1}, 1<k\in \mathbb N).$ Then   there exists $T^{*}=T^{\ast}(\|u_{i}\|_{M^{p,q}} )$ such that \eqref{gnlkgh} has a unique solution $u\in C([0, T^{*}), M^{p,q} (\mathbb R^d)).$ 
\item Assume that $V\in M^{1, \infty}(\R^d)$  and $u_i \in M^{p,1}(\R^d) \ (1\leq p \leq  \infty).$ Then   there exists
   $T^{*}=T^{\ast}(\|u_{i}\|_{M^{p,1}} )$ such that \eqref{nlkgh} has a unique solution $u\in C([0, T^{*}), M^{p,q} (\mathbb R^d)).$ 
 \end{enumerate} 
\end{Theorem}
\begin{proof}
Taking Proposition \ref{wp} and  \cite[Lemmas 4.8 and  4.9]{bhimani2018nonlinear}  and \cite[Lemmas 4.2 and 4.3]{manna2018cauchy} into account, the standard fixed point argument gives the desired result.  We will omit the details.
\end{proof}
\begin{Remark}\label{fr} The analogue of Theorem \ref{wht1} is true for equations \eqref{nlw} and \eqref{fHTE}.
\end{Remark}

\noindent
\textbf{Acknowledgments.}  DGB is very grateful to Professor  Kasso Okoudjou for hosting and  arranging research facilities at the University of Maryland.  DGB is  thankful to  SERB Indo-US Postdoctoral Fellowship (2017/142-Divyang G Bhimani) for the financial support. DGB is also thankful to DST-INSPIRE and TIFR CAM for the  academic leave.

\bibliographystyle{amsplain}
\bibliography{nlkgh}

\providecommand{\bysame}{\leavevmode\hbox to3em{\hrulefill}\thinspace}
\providecommand{\MR}{\relax\ifhmode\unskip\space\fi MR }
\providecommand{\MRhref}[2]{%
  \href{http://www.ams.org/mathscinet-getitem?mr=#1}{#2}
}
\providecommand{\href}[2]{#2}
\begin{thebibliography}{10}

\bibitem{baoxiang2006isometric}
Wang Baoxiang, Zhao Lifeng, and Guo Boling, \emph{Isometric decomposition
  operators, function spaces $e^{p,q}_{\lambda}$ and applications to nonlinear
  evolution equations}, Journal of Functional Analysis \textbf{233} (2006),
  no.~1, 1--39.

\bibitem{benyi2007unimodular}
{\'A}rp{\'a}d B{\'e}nyi, Karlheinz Gr{\"o}chenig, Kasso~A Okoudjou, and Luke~G
  Rogers, \emph{Unimodular fourier multipliers for modulation spaces}, Journal
  of Functional Analysis \textbf{246} (2007), no.~2, 366--384.

\bibitem{benyi2009local}
{\'A}rp{\'a}d B{\'e}nyi and Kasso~A Okoudjou, \emph{Local well-posedness of
  nonlinear dispersive equations on modulation spaces}, Bulletin of the London
  Mathematical Society \textbf{41} (2009), no.~3, 549--558.

\bibitem{Bhimani2018global}
Divyang~G Bhimani, \emph{Global well-posedness for fractional {H}artree
  equation on modulation spaces and fourier algebra}, arXiv:1810.04076.

\bibitem{bhimani2016cauchy}
\bysame, \emph{The cauchy problem for the {H}artree type equation in modulation
  spaces}, Nonlinear Analysis \textbf{130} (2016), 190--201.

\bibitem{bhimani2018nonlinear}
\bysame, \emph{The nonlinear {S}chr{\"o}dinger equations with harmonic
  potential in modulation spaces}, arXiv preprint arXiv:1810.06556 (2018).

\bibitem{bhimani2016functions}
Divyang~G Bhimani and PK~Ratnakumar, \emph{Functions operating on modulation
  spaces and nonlinear dispersive equations}, Journal of Functional Analysis
  \textbf{270} (2016), no.~2, 621--648.

\bibitem{cazenave2003semilinear}
Thierry Cazenave, \emph{Semilinear {S}chr{\"o}dinger equations}, vol.~10,
  American Mathematical Soc., 2003.

\bibitem{chen2012asymptotic}
Jiecheng Chen, Dashan Fan, and Lijing Sun, \emph{Asymptotic estimates for
  unimodular fourier multipliers on modulation spaces}, Discrete \& Continuous
  Dynamical Systems-A \textbf{32} (2012), no.~2, 467--485.

\bibitem{chen2017non}
JieCheng Chen, Qiang Huang, and XiangRong Zhu, \emph{Non-integer power estimate
  in modulation spaces and its applications}, Science China Mathematics
  \textbf{60} (2017), no.~8, 1443--1460.

\bibitem{cheng2012small}
Xing Cheng and Yanfang Gao, \emph{Small data global well-posedness for the
  nonlinear wave equation with nonlocal nonlinearity}, Mathematical Methods in
  the Applied Sciences \textbf{1} (2012), no.~36, 99--112.

\bibitem{cho2013cauchy}
Yonggeun Cho, Hichem Hajaiej, Gyeongha Hwang, and Tohru Ozawa, \emph{On the
  cauchy problem of fractional schr{\"o}dinger equation with hartree type
  nonlinearity}, Funkcialaj Ekvacioj \textbf{56} (2013), no.~2, 193--224.

\bibitem{cordero2009remarks}
Elena Cordero and Fabio Nicola, \emph{Remarks on fourier multipliers and
  applications to the wave equation}, Journal of Mathematical Analysis and
  Applications \textbf{353} (2009), no.~2, 583--591.

\bibitem{feichtinger1983modulation}
Hans~G Feichtinger, \emph{Modulation spaces on locally compact abelian groups},
  Universit{\"a}t Wien. Mathematisches Institut, 1983.

\bibitem{grochenig2013foundations}
Karlheinz Gr{\"o}chenig, \emph{Foundations of time-frequency analysis},
  Springer Science \& Business Media, 2013.

\bibitem{guo2014stability}
Weichao Guo and Jiecheng Chen, \emph{Stability of nonlinear {S}chr{\"o}dinger
  equations on modulation spaces}, Frontiers of Mathematics in China \textbf{9}
  (2014), no.~2, 275--301.

\bibitem{hidano2000small}
Kunio Hidano, \emph{Small data scattering and blow-up for a wave equation with
  a cubic convolution}, Funkcialaj Ekvacioj. Serio Internacia \textbf{43}
  (2000), no.~3, 559--588.

\bibitem{kato2016solutions}
Tomoya Kato et~al., \emph{Solutions to nonlinear higher order schr{\"o}dinger
  equations with small initial data on modulation spaces}, Advances in
  Differential Equations \textbf{21} (2016), no.~3/4, 201--234.

\bibitem{kobayashi2011inclusion}
Masaharu Kobayashi and Mitsuru Sugimoto, \emph{The inclusion relation between
  sobolev and modulation spaces}, Journal of Functional Analysis \textbf{260}
  (2011), no.~11, 3189--3208.

\bibitem{manna2017modulation}
Ramesh Manna, \emph{Modulation spaces and non-linear {H}artree type equations},
  Nonlinear Analysis \textbf{162} (2017), 76--90.

\bibitem{manna2018cauchy}
\bysame, \emph{The cauchy problem for non-linear higher order {H}artree type
  equation in modulation spaces}, J Fourier Anal Appl,
  https://doi.org/10.1007/s00041-018-9629-z (2018), 1--31.

\bibitem{menzala1982wave}
G~Perla Menzala and Walter~A Strauss, \emph{On a wave equation with a cubic
  convolution}, Journal of Differential Equations \textbf{43} (1982), no.~1,
  93--105.

\bibitem{miao2007global}
Changxing Miao, Guixiang Xu, and Lifeng Zhao, \emph{Global well-posedness and
  scattering for the energy-critical, defocusing {H}artree equation for radial
  data}, Journal of Functional Analysis.

\bibitem{miao2011global}
Changxing Miao and Junyong Zhang, \emph{On global solution to the
  {K}lein--{G}ordon--{H}artree equation below energy space}, Journal of
  Differential Equations \textbf{250} (2011), no.~8, 3418--3447.

\bibitem{miao2015scattering}
Changxing Miao, Junyong Zhang, and Jiqiang Zheng, \emph{Scattering theory for
  the radial $\dot{H}^{1/2}-$critical wave equation with a cubic convolution},
  Journal of Differential Equations \textbf{259} (2015), no.~12, 7199--7237.

\bibitem{miao2014energy}
Changxing Miao and Jiqiang Zheng, \emph{Energy scattering for a
  {K}lein--{G}ordon equation with a cubic convolution}, Journal of Differential
  Equations \textbf{257} (2014), no.~6, 2178--2224.

\bibitem{mochizuki1989small}
Kiyoshi Mochizuki, \emph{On small data scattering with cubic convolution
  nonlinearity}, Journal of the Mathematical Society of Japan \textbf{41}
  (1989), no.~1, 143--160.

\bibitem{ruzhansky2012modulation}
Michael Ruzhansky, Mitsuru Sugimoto, and Baoxiang Wang, \emph{Modulation spaces
  and nonlinear evolution equations}, Evolution equations of hyperbolic and
  Schr{\"o}dinger type, Springer, 2012, pp.~267--283.

\bibitem{tsutaya2014scattering}
Kimitoshi Tsutaya, \emph{Scattering theory for the wave equation of a {H}artree
  type in three space dimensions}, Discrete \& Continuous Dynamical Systems-A
  \textbf{34} (2014), no.~5, 2261--2281.

\bibitem{wang2007global}
Baoxiang Wang and Henryk Hudzik, \emph{The global cauchy problem for the {NLS}
  and {NLKG} with small rough data}, Journal of Differential Equations
  \textbf{232} (2007), no.~1, 36--73.

\bibitem{wang2011harmonic}
Baoxiang Wang, Zhaohui Huo, Zihua Guo, and Chengchun Hao, \emph{Harmonic
  analysis method for nonlinear evolution equations, i}, World Scientific,
  2011.

\bibitem{zhao2014klein}
Guoping Zhao, Jiecheng Chen, and Weichao Guo, \emph{Klein-{G}ordon equations on
  modulation spaces}, Abstract and Applied Analysis, vol. 2014, Hindawi, 2014.

\end{thebibliography}

\end{document}